\documentclass[12pt,oneside]{amsart}
\usepackage{amssymb, amsmath, amsthm}
\usepackage{fullpage}

\usepackage{epsfig}
\usepackage{graphicx}
\usepackage{color}
\usepackage{enumerate}
\usepackage{amsrefs}
\usepackage{pinlabel}

\theoremstyle{plain}
\newtheorem{theorem}{Theorem}[section]

\newtheorem*{theorem*}{Theorem}
\newtheorem*{firstthm-intro}{Theorem \ref{firstthm}}
\newtheorem*{reduciblethm-intro}{Theorem \ref{reduciblethm}}
\newtheorem*{primethm-intro}{Theorem \ref{primethm}}

\newtheorem{corollary}[theorem]{Corollary}
\newtheorem{lemma}[theorem]{Lemma}

\theoremstyle{definition}
\newtheorem{remark}[theorem]{Remark}
\newtheorem{definition}[theorem]{Definition}

\newtheorem*{example}{Example}

\theoremstyle{definition}

\newcommand{\C}{\mathbb C}

\renewcommand{\L}{\mathcal L}

\newcommand{\nil}{\varnothing}

\newcommand{\wihat}{\widehat}
\newcommand{\defn}[1]{\emph{#1}}

\newcommand{\boundary}{\partial}

\newcommand{\mc}[1]{\mathcal{#1}}

\newcommand{\nbhd}{N}
\newcommand{\onbhd}{\mathring{\nbhd}}

\begin{document}

   \title[Distance for bridge trisections]{Kirby-Thompson distance for trisections of knotted surfaces}
   \author{Ryan Blair}

   \address{Ryan Blair\\
Department of Mathematics\\
California State University, Long Beach\\
1250 Bellflower Blvd\\
Long Beach, CA 90840}
   \email{ryan.blair@csulb.edu}
   \author{Marion Campisi}
   \address{Marion Campisi \\ San Jos\'e State University\\
One Washington Square\\
San Jos\'e, CA 95192}
\email{marion.campisi@sjsu.edu}
   
   \author{Scott A. Taylor}
   \address{Scott A. Taylor \\ Colby College \\5832 Mayflower Hill \\ Waterville, ME 04901}
   \email{scott.taylor@colby.edu}
   \author{Maggy Tomova}
   \address{Maggy Tomova \\ College of Liberal Arts and Sciences\\ The University of Iowa \\ Iowa City, IA 52242}
   \email{maggy-tomova@uiowa.edu}
   \maketitle
   
   \begin{abstract}
      We adapt work of Kirby-Thompson and Zupan to define an integer invariant $\L(\mathcal{T})$ of a bridge trisection $\mathcal{T}$ of a smooth surface $S$ in $S^4$ or $B^4$.  We show that when $\mc{L}(\mathcal{T})=0$, then the surface $S$ is unknotted. We also show that for a trisection $\mathcal{T}$ of an irreducible surface, bridge number produces a lower bound for  $\mc{L}(\mathcal{T})$. Consequently $\mc{L}$ can be arbitrarily large. %
   \end{abstract}

   
   \section{Introduction}

Inspired by Hempel's distance for Heegaard splittings \cite{Hempel}, Kirby and Thompson \cite{KirbyThompson} recently defined a nonnegative integer valued invariant $\L(X)$ of a smooth 4-manifold $X$. They  show that  when $\L(X)=0$, then $X$ is diffeomorphic to the connected sum of the 4-sphere with some number (possibly zero) of copies of $S^1 \times S^3$, $S^2 \times S^2$, and $\C P^2$.  This was extended to be an invariant for smooth 4-manifolds with boundary in \cite{CIMT}, where it was shown that if the invariant takes the value zero on a rational homology ball, then the rational homology ball is a 4-ball $B^4$.  In this paper, we adapt the definition to apply to smooth surfaces $S$ properly embedded in $S^4$ or $B^4$. We prove:

\begin{firstthm-intro}
Suppose that $S \subset S^4$ is a smooth, closed surface with $\L(S) = 0$. Then $S$ is the distant sum of unknotted 2-spheres and unknotted nonorientable surfaces.
\end{firstthm-intro}

As with the previous incarnations of this theorem, the invariant $\L$ is defined using trisections. Trisections are, in some sense, a 4-dimensional version of Heegaard splittings of 3-manifolds. Trisections of both smooth closed 4-manifolds and smooth compact 4-manifolds with boundary were introduced by Gay and Kirby \cite{GayKirby}. Meier and Zupan \cite{MZ1} adapted this definition, defining bridge trisections for smooth surfaces in $S^4$ (and later in other closed 4-manifolds \cite{MZ2}). They show that every surface embedded in a closed 4-manifold has a bridge trisection. Bridge trisections are a 4-dimensional version of bridge position for links in closed 3-manifolds. We consider only trisection surfaces of genus 0. Associated to each bridge trisection $\mc{T}$ of a smooth surface $S$ in $S^4$ are positive integers: the \defn{bridge number} $b$ and the \defn{patch numbers} $c_1$, $c_2$, and $c_3$.  We say that $\mc{T}$ is a $(b; c_1, c_2, c_3)$-trisection of $S$. The minimal value of $b$ over all possible trisections of $S$ is called the \defn{bridge number} $b(S)$ of $S$. It is the case that $2b + \chi(S)\geq 0$ (where $\chi(S)$ is the Euler characteristic of $S$).

We establish links between these numbers, the invariant $\L$, and the topological structure of both $S$ and its trisections.  To state our results, recall that a smooth, closed surface $S \subset S^4$ is \defn{unknotted} or \defn{trivial} if it is orientable and bounds a 3-dimensional handlebody in $S^4$ or if it is nonorientable and is the connected sum of unknotted projective planes \cite{hosokawa1979}. (A projective plane is unknotted if it is obtained by attaching a disk to a half-twisted M\"obius band in the equatorial $S^3 \subset S^4$.) Meier and Zupan prove that a surface $S$ with $b(S) \leq 3$ is unknotted \cite[Theorem 1.8]{MZ1}. A smooth, closed, connected, orientable surface in $S^4$ is \defn{irreducible} if it is nontrivial and not the connected sum of a nontrivial surface and a trivial surface of genus $g \geq 1$. 

\begin{reduciblethm-intro}
If $S \subset S^4$ is a smooth, closed, connected, orientable, irreducible surface then
\[
\L(S) > b(S) - g(S) - 2,
\]
where $g(S)$ is the genus of $S$.
\end{reduciblethm-intro}

Meier and Zupan \cite[Section 5]{MZ1} show that spun torus knots have arbitrarily large bridge number, and so $\L(S)$ can be arbitrarily large for knotted spheres. Combining their work with that of Livingston \cite{Livingston}), it is likely the case that $\L(S)$ can be arbitrarily large for tori as well. 

Theorem \ref{upperbdthm}, which is a more general version of Theorem \ref{reduciblethm}, gives an upper bound on $\L(S)$ such that if it is satisfied, then $S$ is the nontrivial connected sum of two other surfaces $S_1$ and $S_2$ such that $\L(S_1) + \L(S_2) = \L(S)$. A smooth, closed surface $S \subset S^4$ is \defn{prime} if it is nontrivial and not a connected sum or distant sum of nontrivial smooth surfaces.  Irreducible surfaces need not be prime. In fact, it is unknown if any surface is prime; by \cite{Viro}, there is a nontrivial sphere $S \subset S^4$ whose connected sum with an unknotted projective plane $P$ is equivalent to $P$ (see \cite[Section 2.3.2]{CKS}). We prove:
\begin{primethm-intro}
If $S \subset S^4$ is a surface that is smooth, closed, and prime. Then
\[
\L(S) > 2b(S)  + 2\chi(S)  - 9
\]
\end{primethm-intro}

Finally, we remark that the essence of our results also applies to smooth surfaces with boundary that are properly embedded in $B^4$. Jeffrey Meier \cite{Meier-Relative} has considered the definition of bridge trisections (of arbitrary genus) for properly embedded surfaces in 4-manifolds with boundary. Independent of, but subsequent to, his work, we arrived at an equivalent definition for surfaces in the 4--ball. As the theory is still in its infancy, we hope that providing different, more targeted, exposition is helpful. Meier proves a more general version of:
\begin{theorem*}[Meier]
Every properly embedded smooth surface $S$ in $B^4$ has a genus 0 bridge trisection.
\end{theorem*}
By Meier's theorem, $\L$ provides an invariant arising from trisections for, say, slice disks or other surfaces witnessing the 4-ball genus of knots in $S^3$.

\subsection{Outline}

The remainder of this section establishes terminology, conventions and notation. In Section \ref{trisections}, we recall the definitions and constructions related to bridge trisections of closed surfaces in $S^4$ from \cite{MZ1} and generalize them to the case of surfaces with boundary in $B^4$. Section \ref{sec:surgery} describes how bridge surfaces in 3-dimensions and trisection surfaces in 4-dimensions can be cut open along certain spheres. Section \ref{sec:sums} describes the connection between this operation and connected sums and distant sums.  In Section \ref{pants}, we define the pants complex of a bridge surface or relative bridge surface. We use the pants complex to define the invariant $\L$ and prove a sequence of lemmas relevant to the study of $\L$.  Finally, in Section \ref{sec:reducibletrisection}, we prove our main theorems.

\subsection{Terminology, conventions, and notation}

All manifolds appearing in this paper should be understood to be smooth and with (possibly empty) boundary. If $Y$ is a submanifold of $X$, we let $\nbhd(Y)$ and $\onbhd(Y)$ denote closed and open regular neighborhoods of $Y$ in $X$. The interval $[-1,1]$ is denoted $I$.

A simple closed curve $\gamma$ in a disk or sphere $\Sigma$ with marked points (called \defn{punctures}) is \defn{essential} if it is disjoint from the punctures, does not bound an unpunctured or once-punctured disk in $\Sigma$, and does not cobound an unpunctured annulus in $\Sigma$ with $\boundary \Sigma$.  Isotopies of essential curves are proper ambient isotopies of $\Sigma$ relative to the punctures. We will often conflate the isotopy class of a collection of essential curves with its representatives. In our context, this should cause no confusion. The surface $\Sigma$ is \defn{admissible} if there are at least 3 punctures, when $\Sigma$ is a disk, and at least 4 punctures when $\Sigma$ is a sphere. A connected subsurface $P \subset \Sigma$ is a \defn{pair of pants} if it is either a disk with 2 punctures, an annulus with one puncture, or is unpunctured and has $\chi(P) = -1$. 

Suppose that $M$ is a 3--manifold (in our case, the 3-sphere or 3-ball) and that $T \subset M$ is a properly embedded 1--manifold and $S \subset M$ is an embedded surface transverse to $T$. We consider the points $T \cap S$ to be punctures on $S$. Suppose also that $D \subset M$ is a disk with $\boundary D \subset S$ an essential curve and with interior disjoint from $S$ and transverse to $T$. Then $D$ is a \defn{compressing disk} if $D \cap T = \nil$ and a \defn{cut disk} if $|D \cap T| = 1$. A disk that is a compressing disk or a cut disk is a \defn{c-disk}. If $D$ is a c-disk for $\boundary M$, then we also say that $D$ is a c-disk \defn{in} $(M,T)$. If $\tau$ is a 1-manifold properly embedded in $Z = D^2 \times I$, then an annulus properly embedded in $Z$, disjoint from $\tau$ and with one boundary component a curve on each of $D^2 \times \{\pm 1\}$ is called a \defn{spanning annulus}. For convenience we will often say that one component of the boundary of a spanning annulus \defn{bounds} the annulus (even though it actually cobounds the annulus with the other boundary component). 

\subsection{Acknowledgments} Blair was supported by NSF grant DMS-1821254. Campisi was supported by a San Jos\'e State University RSCA Grant.  Taylor was supported by a Colby College Research Grant. We are grateful to Rom\'an Aranda, Chuck Livingston, Jeffrey Meier, Maggie Miller, Alex Zupan, and the anonymous referee for helpful comments on the manuscript.

\section{Bridge trisections}
\label{trisections}

We first define a (genus 0) trisection of $S^4$ and then a (genus zero) bridge trisection of a smooth surface in $S^4$. After that we move on to the definition of relative trisections.

\begin{remark}
All of the trisections we discuss have trisection surfaces that are either the sphere or the disk. We usually omit the adjective ``genus 0'' in what follows.
\end{remark}

\begin{definition}[Gay-Kirby]
   A \emph{$0$-trisection} of the 4--sphere $S^4$ is a decomposition into 4-balls $S^4=X_1\cup_{\Sigma} X_2\cup_{\Sigma} X_3$ with 
   \begin{enumerate}
   \item $B_{ij}=X_i\cap X_j$, $i, j\in\{1,2,3\}$ and $i\neq j$, a 3-ball, and
   \item $\Sigma=X_1\cap X_2\cap X_3=B_{12}\cap B_{23}\cap B_{31}$ a 2-sphere.
   \end{enumerate}
      \end{definition}
   
A \defn{trivial tangle} $(B, \kappa)$ is a 3-ball $B$ containing properly embedded arcs $\kappa$ such that, fixing the endpoints of $\kappa$, we may isotope $\kappa$ into $\boundary B$.  We consider $\partial B$ to be a $2|\kappa|$-punctured surface, with the points $\boundary \kappa$ being the punctures. For expositional convenience, we also set $\boundary_+ B = \boundary B$ and $\boundary_- B = \nil$. A \defn{bridge splitting} for a knot $K\subset S^3$ is a decomposition $(S^3, K)=(B_1,\alpha_1)\cup_{\Sigma}(B_2, \alpha_2)$ with $(B_i, \alpha_i)$ trivial tangles and $\Sigma=\partial B_i$, for $i=1,2$.  The surface $\Sigma$ is the \defn{bridge sphere} of the splitting. A \defn{trivial disk system}  $(X, \mathcal{D})$ is a 4-ball $X$ containing a collection $\mathcal{D}$ of properly embedded disks which are isotopic, relative to their boundary, into $\boundary X$.     

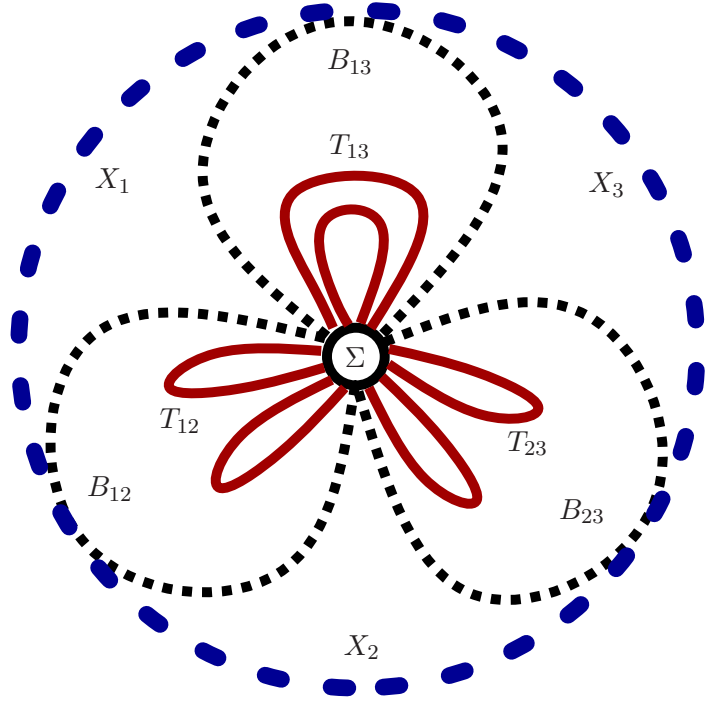
\begin{figure}[ht!]
\labellist
\small\hair 2pt
\pinlabel $\Sigma$ at 260 256
\pinlabel $X_2$ [b] at 265 24
\pinlabel $X_1$ [tl] at 61 401
\pinlabel $X_3$ [tr] at 466 398
\pinlabel $B_{12}$ [bl] at 55 145
\pinlabel $B_{13}$ [t] at 256 492
\pinlabel $B_{23}$ [br] at 451 128
\pinlabel $T_{12}$ [t] at 128 221
\pinlabel $T_{13}$ [b] at 256 402
\pinlabel $T_{23}$ [t] at 390 204
\endlabellist
\centering
\includegraphics[scale=0.5]{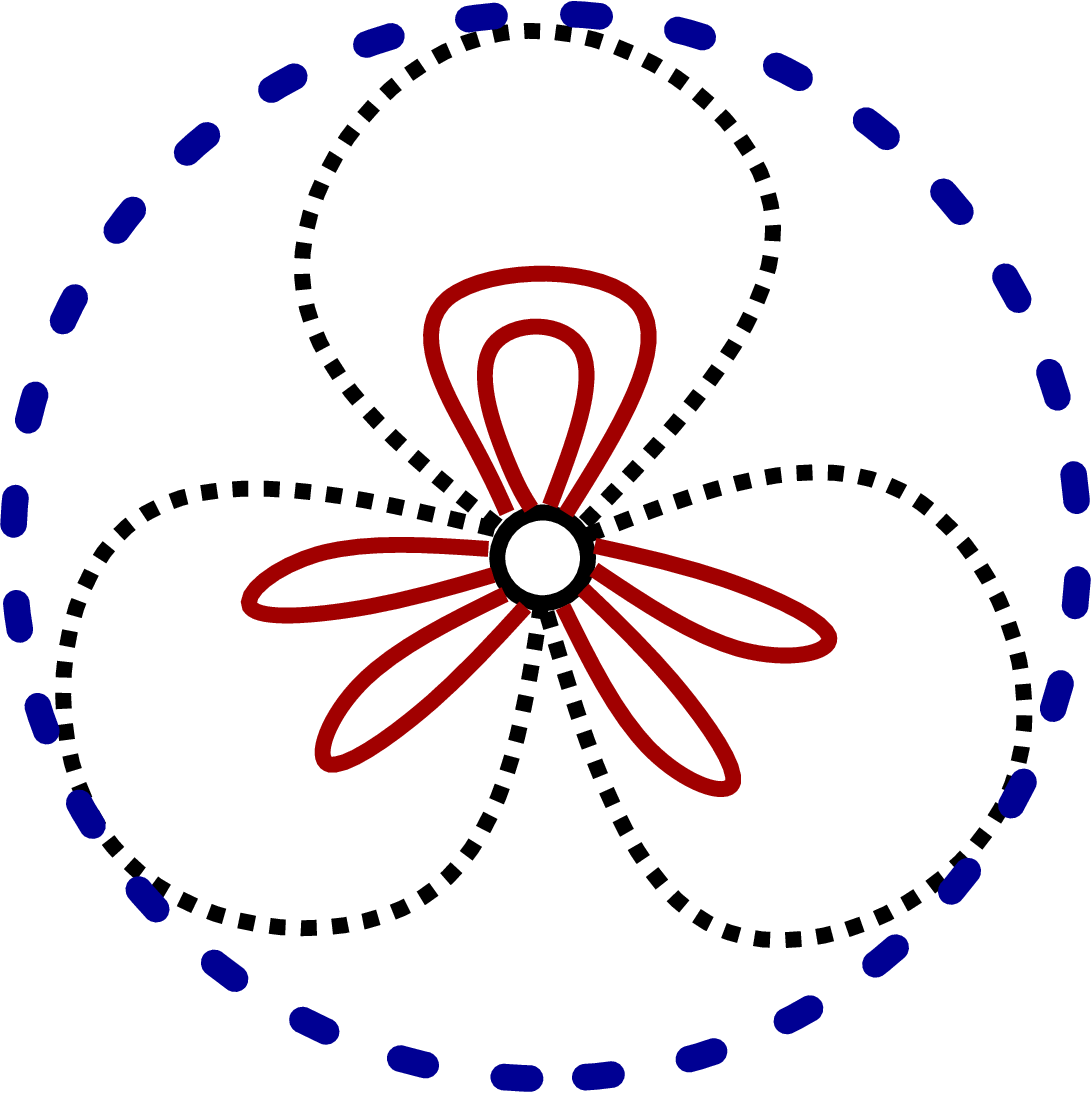}
\caption{A schematic depiction of a bridge trisection. The circle in the center represents the trisection surface, the sphere $\Sigma$. It is the boundary of three 3-balls, depicted with dashed black lines. The union of any two of those is a copy of the 3-sphere which is the boundary of a 4-ball. The union of the 4-balls is $S^4$ and is depicted with the dashed blue circle. In each 3-ball is a trivial tangle, depicted with red arcs. The disks in each 4-ball are not depicted.}
\label{fig:bridgetrisection}
\end{figure}

\begin{definition}[Meier-Zupan]
A \emph{$(b;c_1, c_2, c_3)$-bridge trisection} $\mathcal{T}$ of a  surface $S\subset S^4$ is a 0-trisection $S^4=X_1\cup_{\Sigma}  X_2\cup_{\Sigma}  X_3$  with $b = |S \cap \Sigma|/2$ such that for all $\{i,j,k\} = \{1,2,3\}$, we have:     
    
\begin{enumerate}
\item $(S_j, L_j)=(B_{ij}\cup_{\Sigma} B_{jk}, (B_{ij}\cup_{\Sigma} B_{jk})\cap S)$ is an unlink with $c_j$-components in $S^3$ (henceforth, just \defn{unlink}).
\item $\Sigma$ is a bridge sphere for $(S_j, L_j)$, decomposing it into the trivial tangles  $(B_{ij}, T_{ij})$ and    $(B_{jk}, T_{jk})$.
\item $(X_{i}, \mathcal{D}_{i})=(X_{i}, X_i\cap S)$ is a trivial disk system.
\end{enumerate}
We call $\mathcal{S}=(B_{12},T_{12})\cup(B_{23},T_{23})\cup(B_{31},T_{31})$ the \defn{spine} of the bridge trisection.  We define $B_{ij}=B_{ji}$ and $\alpha_{ij}=\alpha_{ji}$ for all $i,j \in \{1,2,3\}$ such that $i\neq j$. The minimum $b(S)$ of the \defn{bridge number} $b$ of $\mc{T}$ over all bridge trisections $\mc{T}$ of $S$ is called the \defn{bridge number} of $S$. Note that $b$ is a positive integer.
\end{definition}

See Figure \ref{fig:bridgetrisection} for a schematic depiction of a bridge trisection and its spine.  For ease of exposition when we simultaneously handle relative trisections, in the case of a bridge trisection of a surface in $S^4$ as above, we will set $V = \nil$ and also say that
\[
\mathcal{S}=(B_{12},\alpha_{12})\cup(B_{23},\alpha_{23})\cup(B_{31},\alpha_{31}) \cup (V, V \cap \boundary \Sigma)
\]
is the spine of the trisection.

We now turn to trisections of surfaces properly embedded in $B^4$.  The definition is inspired by the definition of Heegaard surfaces for 3-manifolds with boundary, trisections of 4-manifolds with boundary, and of course bridge trisections of surfaces in $S^4$. See Figures \ref{fig:strictlytrivialtangle} and \ref{fig:relativebridgetrisection} for a schematic depiction of the main components.

For a 3-ball $Z$ parameterized as $D^2 \times I$, we define the \defn{positive} (\defn{negative}, resp.) boundary to be $\boundary_\pm Z = D^2 \times \{\pm 1\}$. The \defn{vertical boundary} is $\boundary_v Z = \boundary D^2 \times I$.

\begin{definition}[Gay-Kirby]
   A \emph{$0$-trisection} of $B^4$ is a decomposition into 4-balls $B^4=X_1\cup_{\Sigma} X_2\cup_{\Sigma} X_3$ with 
   \begin{enumerate}
   \item $Z_{ij}=X_i\cap X_j$, $i, j\in\{1,2,3\}$ and $i\neq j$, a 3-ball parameterized as $D^2\times I$, and
   \item $\Sigma=X_1\cap X_2\cap X_3=\boundary_+ Z_{12} = \boundary_+ Z_{23} = \boundary_+ Z_{31}$ a disk.
   \end{enumerate}
      \end{definition}
      
 For later use, we observe the following. In $B^4$, $\nbhd(\Sigma)$ is homeomorphic to $\Sigma \times D^2$. The intersection $\nbhd(\Sigma)$ with $\boundary B^4$ is a solid torus $V' = (\boundary \Sigma) \times D^2 \subset \boundary B^4$. The complementary solid torus $V$ which is the closure of $\boundary B^4 \setminus V'$ contains the three disks $\boundary_- Z_{12}$, $\boundary_- Z_{23}$, and $\boundary_- Z_{13}$ as properly embedded meridian disks. Similarly, for $\{i,j,k\} = \{1,2,3\}$, the 3-sphere $\boundary X_i$ is decomposed into two solid tori $V_i$ and $V'_i$. The solid torus $V_i$ is the union of the 3-balls $Z_{ij}$ and $Z_{ik}$ along $\Sigma$ together with the component of $V \setminus (\boundary_- Z_{12} \cup \boundary_- Z_{23} \cup \boundary_- Z_{13})$ between $\boundary_- Z_{ij}$ and $\boundary_- Z_{ik}$. The solid torus $V'_i$ is the closure of $\boundary X_i \setminus V_i$. The solid torus $V'$ is formed by taking the union of the solid tori $V'_1$, $V'_2$, and $V'_3$ along the annuli $\boundary_v Z_{12}$, $\boundary_v Z_{23}$, and $\boundary_v Z_{13}$.
      
A \defn{relative tangle} $(Z,T)$  is a 3-ball $Z = D^2 \times I$  containing a properly embedded 1-manifold $T$ such that $\boundary T$ is contained in the interior of $\boundary_- Z \cup \boundary_+ Z$. It is \defn{trivial} if there is a properly embedded arc $\alpha \subset \boundary_+ Z$ (necessarily containing the punctures) such that $T$ can be isotoped, relative to $\boundary T$, into $\alpha \times I$. The arc $\alpha$ is called a \defn{trace arc}. A relative tangle is \defn{strictly trivial} if $T$ is trivial and has no closed components and no components with both endpoints on $\boundary_- Z$. A relative tangle is \defn{spanning}, if each arc component of $T$ has one endpoint on each of $\boundary_- Z$ and $\boundary_+ Z$. Spanning relative tangles may have closed components. If $T$ is a trivial tangle, a component of $T$ is a \defn{vertical arc} if it has an endpoint on each of $\boundary_\pm Z$ and a \defn{bridge arc} if it has both endpoints on $\boundary_+ Z$. Figure \ref{fig:strictlytrivialtangle} shows an example of a strictly trivial relative tangle. 

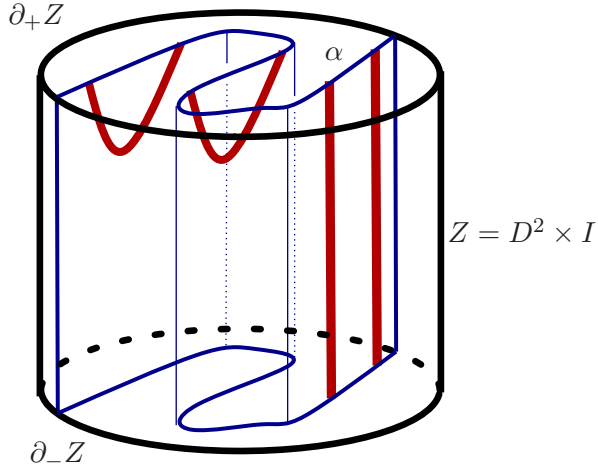
\begin{figure}[ht!]
\labellist
\small\hair 2pt
\pinlabel {$\boundary_+ Z$} [br] at 24 317
\pinlabel {$\boundary_- Z$} [tr] at 40 12
\pinlabel $\alpha$ [br] at 234 295
\pinlabel {$Z = D^2 \times I$} [l] at 306 168
\endlabellist
\centering
\includegraphics[scale=0.5]{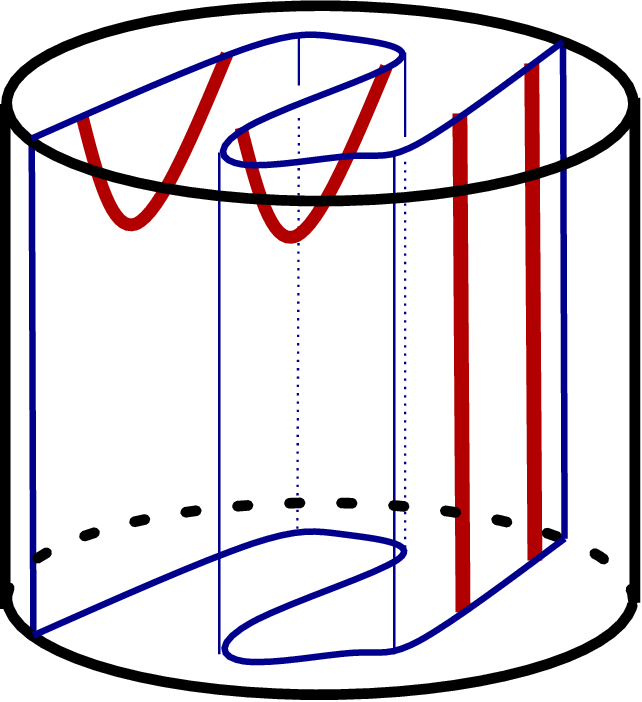}
\caption{An example of a strictly trivial relative tangle $(Z,T)$. It has two bridge arcs and two vertical arcs. One example of a trace arc $\alpha$ along with the disk $\alpha \times I$ are shown.}
\label{fig:strictlytrivialtangle}
\end{figure}

A disk $\Sigma = D^2 \times \{t_0\}$ for some $t_0 \in I\setminus \boundary I$ is a \defn{relative bridge surface} for a relative tangle $(Z,T)$ (with $Z = D^2 \times I$) if the closure of each component of $(Z,T)\setminus \Sigma$ is a strictly trivial tangle with positive boundary $\Sigma$. If $(B_1,\alpha_1)$ and $(B_2, \alpha_2)$ are the closures of the components of $(Z,T)\setminus \Sigma$, we also say that 
$(B_1,\alpha_1)\cup_\Sigma (B_2, \alpha_2)$ is a \defn{relative bridge splitting} of $(Z,T)$.

A link in a solid torus $S^1 \times D^2$ is an $n$-braid if its winding number is $n$ and if it can be isotoped so that the restriction of the map $S^1 \times D^2 \to S^1$ to the link is monotonic. See Figure \ref{fig:relativebridgetrisection} for a schematic depiction of the following definition. 

 \begin{definition}\label{relative bridge}
 A \defn{(genus 0 relative) bridge trisection} $\mc{T}$ of a properly embedded surface $S \subset B^4$ is a 0-trisection $B^4 = (W_1) \cup_\Sigma (W_2) \cup_\Sigma  (W_3)$ such that for all $\{i,j,k\} = \{1,2,3\}$, we have:
\begin{enumerate}
\item $\boundary S \subset S^3 = \boundary B^4$ is an $n$-braid in the solid torus $V$ that is the exterior of $\boundary \Sigma \subset S^3$ where $n = |S \cap \boundary_- Z_{ij}|$
\item $(Z_j, T_j) = (Z_{ij} \cup_\Sigma Z_{jk}, (Z_{ij} \cup_\Sigma Z_{jk})\cap S)$ is a trivial spanning relative tangle.
\item $\Sigma$ is a bridge disk for $(Z_j, T_j)$, decomposing $(Z_j, T_j)$ into the strictly trivial relative tangles $(Z_{ij}, T_{ij})$ and $(Z_{jk}, T_{jk})$.
\item In the solid torus $V_i$, each component of $S \cap V_i$ is either isotopic to a core of $V_i$ or is disjoint from $\boundary_- Z_{ij}$
\item In the 3-sphere $\partial W_i$, the link $S \cap \partial W_i$ is an unlink.
\item $(W_i, \mc{D}_i)=(W_i, W_i\cap S)$ is a trivial disk system.
\end{enumerate}
We call $\mathcal{S}=(Z_{12},T_{12})\cup_\Sigma (Z_{23},T_{23})\cup_\Sigma(Z_{31},T_{31}) \cup (V, V\cap \boundary S)$ the \defn{spine} of the bridge trisection. The disk $\Sigma$ is the \defn{trisection surface} of the trisection and $b(\Sigma) = |\Sigma \cap S|/2$ is the \defn{bridge number of the trisection}. Note that it is a positive integer or half-integer.
\end{definition}

\begin{figure}[ht!]
\labellist
\small\hair 2pt
\pinlabel {$\Sigma$} [b] at 383 429
\pinlabel {$\boundary_+ Z_{12}$} [l] at 441 341
\pinlabel {$\boundary_+ Z_{13}$} [t] at 296 388
\pinlabel {$\boundary_+ Z_{23}$} [t] at 469 388
\pinlabel {$\boundary_- Z_{12}$} [t] at 382 127
\pinlabel {$\boundary_- Z_{13}$} [r] at 137 560
\pinlabel {$\boundary_- Z_{23}$} [l] at 633 560
\endlabellist
\centering
\includegraphics[scale=0.45]{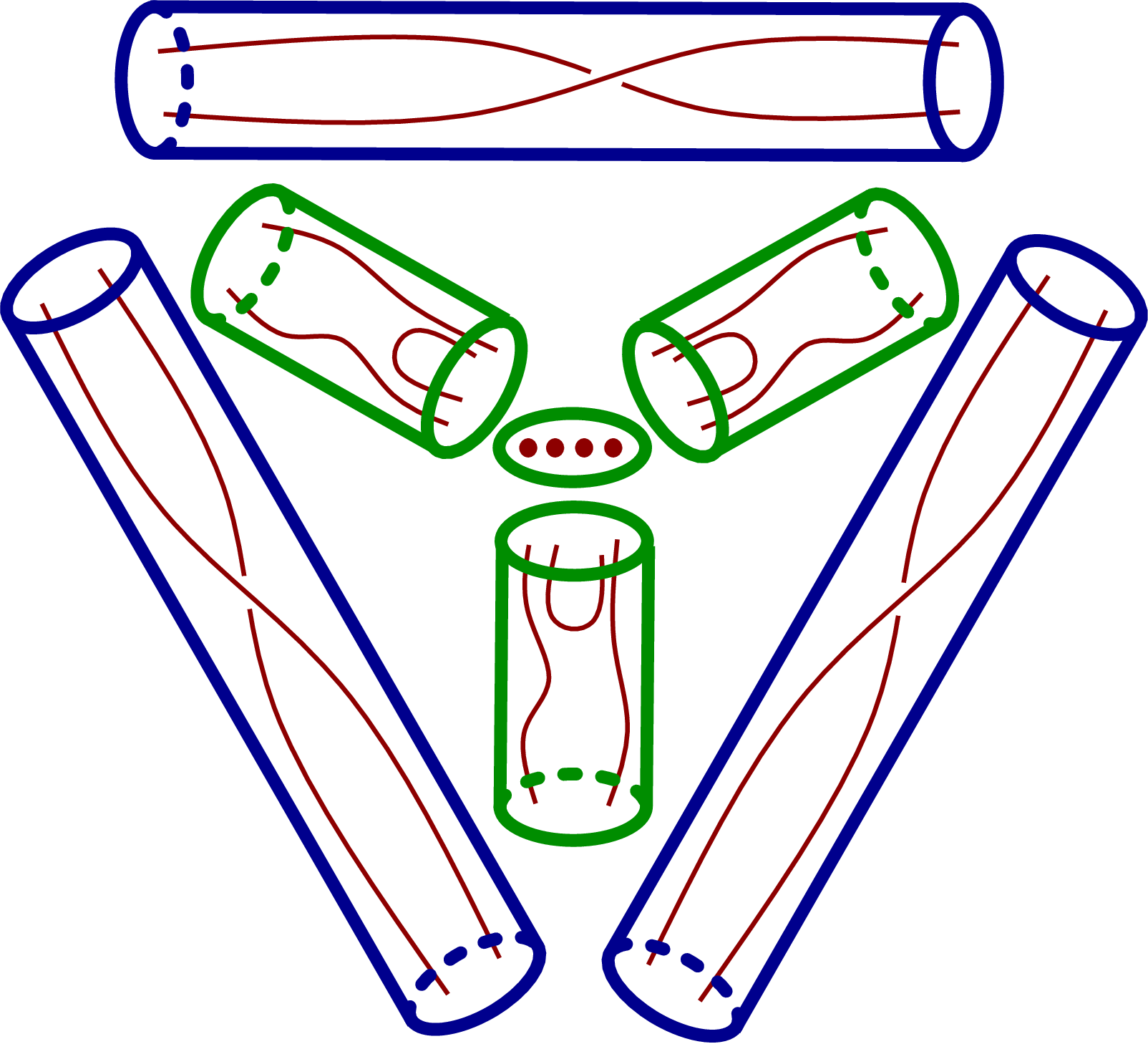}
\caption{A schematic depiction of the spine of a relative bridge trisection, shown without the gluings, without the 4-dimensional pieces, and without a solid torus in the boundary of the 4-ball. The outer three trivial tangles (in blue) form the solid torus $V$ and the braid $(V \cap \boundary S)$. The three inner tangles are the strictly trivial tangles $(Z_{ij}, T_{ij})$. The trisection disk is in the center. }
\label{fig:relativebridgetrisection}
\end{figure}
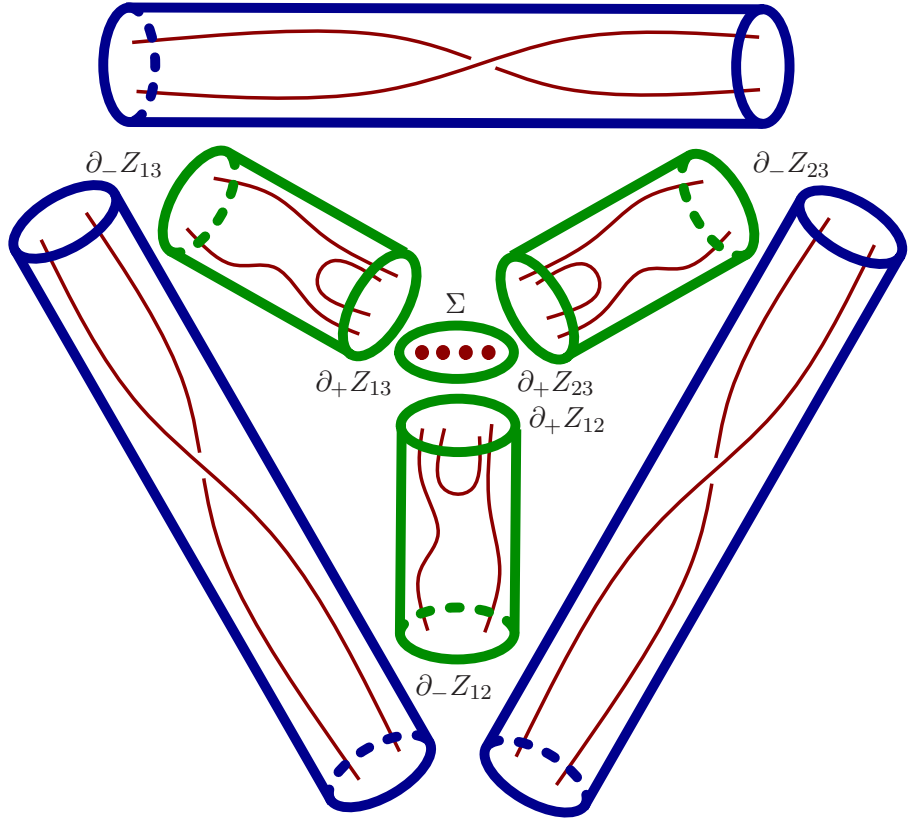

\begin{example}
In Figure \ref{fig:slicedisk}, we show how to plug diagrams of certain tangles into a template to create a bridge trisection for a slice disk for the square knot. When gluing the tangles $Z_{13}$ and $Z_{23}$ together in the process of verifying that the conditions in Definition \ref{relative bridge} are satisfied, remember to take the mirror image of one of the indicated diagrams, as in \cite{MZ1}. Since $Z_{12}$ has no crossings, we do not need to bother mirroring when it is one of the tangles being glued. To verify that the resulting bridge trisection is a trisection of a disk, we can compute the Euler characteristic. In general, if $\boundary S$ is an $n$-braid and if there are $b'$ bridge arcs in each tangle $Z_{ij}$ and $c_{j}$ closed components of $T_j$, we have $\chi(S) = c_1 + c_2 + c_3 + n - b'$. In our case, $n = 3$, $b' = 2$, and each $c_j = 0$.
\end{example}

\begin{figure}[ht!]
\labellist
\small\hair 2pt
\pinlabel {$\Sigma$} [t] at 147 316
\pinlabel {$Z_{12}$} at 148 236
\pinlabel {$Z_{13}$} at 95 320
\pinlabel {$Z_{23}$} at 197 320
\pinlabel {$1$} [t] at 129 292
\pinlabel {$7$} [t] at 166 292
\pinlabel \rotatebox{-60}{$A$} at 45 244
\pinlabel \rotatebox{60}{$B$} at 250 244
\pinlabel {$C$} at 143 409
\pinlabel {$A$} [b] at 56 131
\pinlabel {$B$} [b] at 146 131
\pinlabel {$C$} [b] at 235 131
\pinlabel {$Z_{13}$} [tl] at 17 96
\pinlabel {$Z_{23}$} [tl] at 106 96
\pinlabel {$Z_{12}$} [t] at 216 96
\pinlabel {$1$} [t] at 23 10
\pinlabel {$2$} [t] at 41 10
\pinlabel {$3$} [t] at 47 10
\pinlabel {$4$} [t] at 62 10
\pinlabel {$5$} [t] at 71 10
\pinlabel {$6$} [t] at 80 10
\pinlabel {$7$} [t] at 91 10
\pinlabel {$1$} [t] at 114 10
\pinlabel {$2$} [t] at 123 10
\pinlabel {$3$} [t] at 137 10
\pinlabel {$4$} [t] at 148 10
\pinlabel {$5$} [t] at 157 10
\pinlabel {$6$} [t] at 164 10
\pinlabel {$7$} [t] at 182 10
\pinlabel {$1$} [t] at 200 10
\pinlabel {$2$} [t] at 207 10
\pinlabel {$3$} [t] at 225 10
\pinlabel {$4$} [t] at 234 10
\pinlabel {$5$} [t] at 240 10
\pinlabel {$6$} [t] at 258 10
\pinlabel {$7$} [t] at 267 10
\endlabellist
\centering
\includegraphics[scale=1.0]{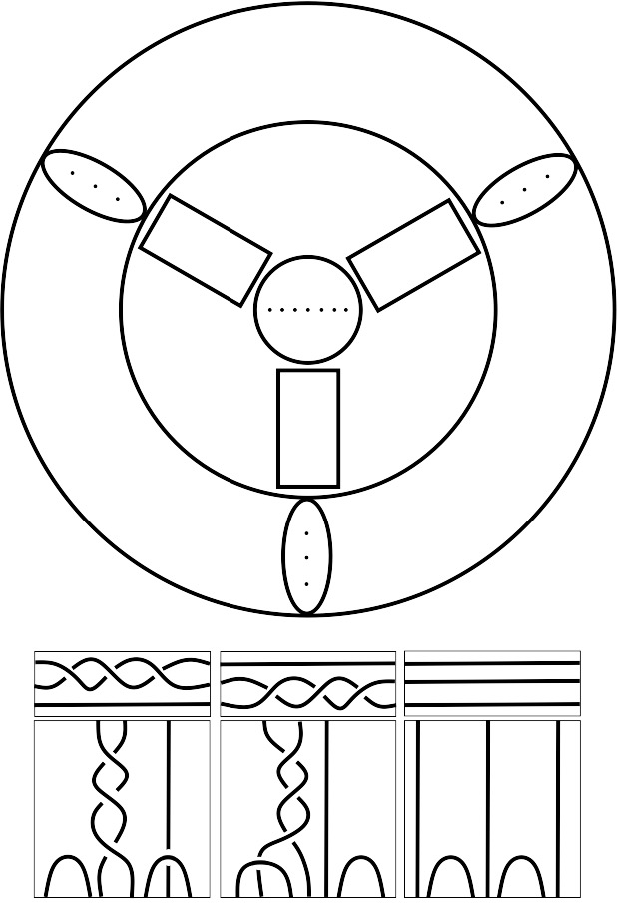}
\caption{A relative bridge trisection of a slice disk for the square knot. This particular trisection was found by \emph{ad hoc} methods. Meier \cite{Meier-Relative} gives a systematic approach.}
\label{fig:slicedisk}
\end{figure}

The proof of \cite[Lemma 2.5]{MZ1} shows that, even in the relative case, we have the following result.

\begin{lemma}[Lemma 2.5 of \cite{MZ1}]\label{MZ-spinedet}
\label{lemma:spine}
If two surfaces in $S^4$ or $B^4$ have trisections with the same spine, then they are isotopic to each other by a smooth proper isotopy relative to the spine.
\end{lemma}
As a consequence, we consider two trisections to be equivalent if their spines are smoothly properly isotopic via isotopies taking the trisection surfaces to each other.

Finally, we need the following lemma concerning spines without bridge arcs.

\begin{lemma}\label{only arcs}
Suppose that $\mc{T}$ is a relative bridge trisection for a surface $S \subset B^4$ with spine $(Z_{12},T_{12})\cup_\Sigma (Z_{23},T_{23})\cup_\Sigma(Z_{31},T_{31}) \cup (V, \boundary S \cap V)$. Suppose also that each $T_{ij}$ consists only of vertical arcs. Then $S$ is the union of $\boundary$-parallel disks, each having bridge number 1/2 with respect to $\Sigma$.
\end{lemma}
\begin{proof}
Suppose that $p \in \Sigma \cap S$ is a puncture. Since each $T_{ij}$ consists only of vertical arcs, $p$ is incident to a vertical arc $\tau_{ij}(p)$ in $T_{ij}$. Since in each solid torus $V_i$, the components of the link  $S \cap V_i$ intersecting $V$ is an $n$-braid where $n$ is the number of components of $S \cap V_i$, in $V_i$, there is an arc $\tau_i(p) \subset (S \cap V_i\cap V)$ joining the endpoint of $\tau_{ij(p)}$ on $\boundary_- Z_{ij}$ to the endpoint of $\tau_{ik}(p)$ on $\boundary_- Z_{ik}$. Observe that $L_i(p) = \tau_{ij}(p) \cup \tau_i(p) \cup \tau_{ik}(p)$ is a component of $S \cap V_i$ and, thus, is isotopic to a core of $V_i$. By the definition of bridge trisection it bounds a properly embedded disk $D_i(p)$ in the 4-ball $W_i$. This disk is a component of $S \cap W_i$. Finally, observe that $D_1(p) \cup D_2(p) \cup D_3(p)$ is a disk component $D(p)$ of $S$ and that it has bridge number 1/2 with respect to $\Sigma$. Also, $D(p)$ is $\boundary$-parallel, as can be seen by piecing together the $\boundary$-parallelism of each third. Since this is true for all $p \in \Sigma \cap S$, the result holds.
\end{proof}

\section{Surgery on spheres}\label{sec:surgery}

In this section, by way of establishing terminology we review the connected sum of manifolds, knots and bridge disks/spheres and establish lemmas that will be of importance later. The next section extends this to trisections.

In general, if $X$ is a smooth $n$-manifold and $Y \subset X$ is a smooth submanifold diffeomorphic to the sphere $S^{n-1}$, we can perform surgery on $X$ along $Y$, by cutting $X$ open along $Y$ and patching in two copies of the $n$-ball $B^n$ via homeomorphisms from $\boundary B^n$ to the two copies of $Y$ in the boundary of the cut-open manifold. Several cases will be of particular interest:
\begin{enumerate}
\item  If $L$ is a properly embedded 1--manifold in a 3-manifold $M$ and if $P \subset M$ is a separating 2-sphere that is disjoint from $L$, then surgery along $P$ produces two distinct 3-manifolds $M_1$ and $M_2$, obtained from $M$ by cutting open along $P$ and pasting two copies $B_1$ and $B_2$ of a 3-ball. The 3-manifolds $M_1$ and $M_2$ contain 1--manifolds $L_1 \subset M_1$ and $L_2 \subset M_2$ that are the (possibly empty) union of components of $L$. 
\item  If $L$ is a properly embedded 1--manifold in a 3-manifold $M$ and if $P \subset M$ is a separating 2-sphere that intersects $L$ in exactly two points (i.e. a 0-sphere), then again we obtain distinct 3-manifolds $M_1$ and $M_2$ containing 1--manifolds $L_1$ and $L_2$. However, in this case $L_1 \cup L_2$ is obtained from $L$ by cutting open $L$ along $P \cap L$ and pasting in two boundary-parallel intervals properly embedded in $B_1$ and $B_2$. 
\item If $M= S^3$ or $D^2 \times I$ and $\Sigma \subset M$ is a bridge sphere or disc for a 1-manifold $L \subset M$ and if $P$ intersects $\Sigma$ in a simple closed curve $s$, then surgery along $P$ also produces surfaces $\Sigma_1 \subset M_1$ and $\Sigma_2 \subset M_2$ obtained by cutting open $\Sigma$ along $s$ and pasting disks properly embedded in $B_1$ and $B_2$. If $P$ intersects $L$ in two points, we choose those disks to intersect in a single point each of the boundary-parallel intervals in $B_1$ and $B_2$ that are glued to $L \setminus P$ to form $L_1$ and $L_2$. It is easy to check that each of $\Sigma_1$ and $\Sigma_2$ is a bridge sphere or disk for $L_1$ and $L_2$. See Figure \ref{L1andL2}.
\end{enumerate}

In any of the above cases, if $|P \cap L| = 0$, we say that $P$ is an \defn{unpunctured summing sphere}. If $|P \cap L| = 2$ we say that $P$ is a \defn{twice-punctured summing sphere}.  In the third situation, a summing sphere where the curve $s$ is essential in $\Sigma$ is called a \defn{reducing sphere} for $\Sigma$.

   \begin{figure}[ht!]
\labellist
\small\hair 2pt
\pinlabel $\Sigma$ at 150 199
\pinlabel $s$ at 185 215
\pinlabel $\Sigma_1$ at 35 65
\pinlabel $\Sigma_2$ at 273 67
\endlabellist
\centering
\includegraphics[width=5in]{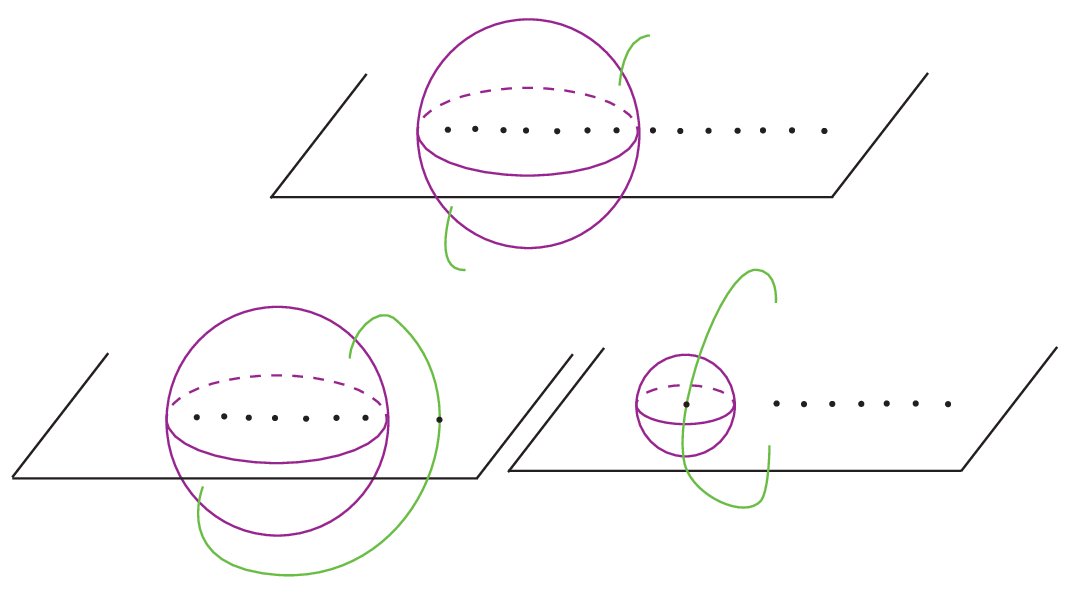}
\caption{Above: A twice-punctured summing sphere $P$ that is a reducing sphere for $\Sigma$. Below: The bridge spheres or disks $\Sigma_1$ and $\Sigma_2$ resulting from surgery on $P$.}
\label{L1andL2}
\end{figure}
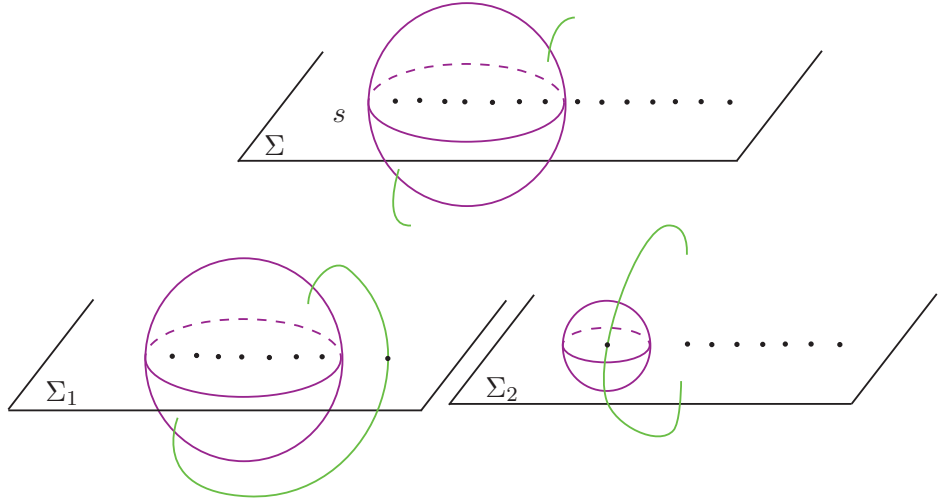

We can generalize the construction slightly by also surgering along annuli. This gives a fourth case of particular interest:
\begin{enumerate}
\item[(4)] Suppose that $(Z,L)$ is a relative spanning tangle with $Z = D^2 \times I$ having bridge disk $\Sigma$ and that $P \subset Z$ is a properly embedded annulus disjoint from $(\boundary D^2) \times I$ and intersecting $\Sigma$ in a single simple closed curve $s$ essential in both $\Sigma$ and $P$. Assume also that $P$ is disjoint from $L$. We call $P$ a \defn{reducing annulus}. It is \defn{even} if it bounds a solid cylinder in $Z$ containing an even number of arcs. A reducing annulus $P$ can be isotoped to be vertical in the product structure on $Z$. 
\end{enumerate}

We can perform surgery on a reducing annulus $P$ as follows. Cut $Z$ open along $P$ to obtain $Z'_1$ and $Z_2$. Without loss of generality, assume that $\boundary \Sigma \subset Z'_1$. Then $\Sigma$ intersects $Z'_1$ in a properly embedded annulus $\Sigma'_1$ and $Z_2$ in a properly embedded disk $\Sigma_2$. Observe that $\Sigma_2$ is a bridge disk for $(Z_2, L \cap Z_2)$. To $s \times I \subset \boundary Z'_1$, glue a copy $B$ of $D^2 \times I$ via a homeomorphism of $(\boundary D^2) \times I$ to $s \times I$ making the fibers match. This creates $Z_1$. In $B$, choose a properly embedded disk whose boundary is glued to $s \subset Z'_1$ and let $\Sigma_1$ be the union of $\Sigma'_1$ with that disk. Then $\Sigma_1$ is a bridge disk for $(Z_1, L \cap Z_1)$.

Recall that the \defn{bridge number} of a bridge surface for a properly embedded 1--manifold is equal to half the number of punctures. In either situation (3) or (4), let $b$ be the bridge number of $\Sigma$ and $b_i$ be the bridge number of $\Sigma_i$ for $i = 1,2$. Observe that if $P$ is a twice-punctured summing sphere, then $b_1 + b_2 = b + 1$; otherwise $b_1 + b_2 = b$. In any case, if $P$ is a reducing sphere (i.e. $s$ is an essential curve on $\Sigma$), then $b_1, b_2 < b$. Similarly, let $c_i$ be the number of closed components of $L_i$ and let $v_i$ be the number of arc components of $L_i$. Let $c$ and $v$ be the number of closed components  and arc components of $L$, respectively. If $P$ is a twice-punctured sphere, then $(c_1 + v_1/2) + (c_2 + v_2/2) = (c + v/2) + 1$; otherwise, since $P$ is an unpunctured summing sphere or annulus,  $(c_1 + v_1/2) + (c_2 + v_2/2) = (c + v/2)$. 

Expanding our viewpoint, suppose $W$ is either $B^4$ or $S^4$ and that $S \subset W$ and $\Sigma \subset W$ are smooth surfaces. Suppose $P \subset W$ is a smooth 3-sphere that is either disjoint from $S$ or intersects $S$ in a single simple closed curve and similarly for $\Sigma$. Then we can surger $W$, $S$, and $\Sigma$ along $P$ simultaneously to obtain 4-manifolds $W_1$ and $W_2$ and properly embedded surfaces $S_1, \Sigma_1 \subset W_1$ and $S_2, \Sigma_2 \subset W_2$. One of $W_1$ or $W_2$ is $S^4$ and the other is either $S^4$ or $B^4$. The next section explores the situation when $\Sigma$ is a trisection surface for $S$. First, however, we establish a few lemmas.

\begin{lemma}\label{consistent bounding}
Suppose that $(Z,L)$ is either a link in $S^3$ or a  spanning relative tangle with bridge sphere or disk $\Sigma$ dividing $(Z,L)$ into tangles $(Z_x, L_x)$ and $(Z_y, L_y)$. Suppose also that $s\subset \Sigma$ is an essential curve such that in each of $(Z_x, L_x)$ and $(Z_y, L_y)$ $s$ bounds either a c-disk or (with a curve in $\boundary_- Z_x$ or $\boundary_- Z_y$) a spanning annulus disjoint from $L$. Then one of the following holds:
\begin{enumerate}
\item In each of $(Z_x, L_x)$ and $(Z_y, L_y)$, the curve $s$ bounds a compressing disk.
\item In each of $(Z_x, L_x)$ and $(Z_y, L_y)$, the curve $s$ bounds a cut disk.
\item In each of $(Z_x, L_x)$ and $(Z_y, L_y)$, the curve $s$ bounds a spanning annulus.
\end{enumerate}
Furthermore, it is not the case that both (1) and (2) hold. If, in $(Z_x, L_x)$ or $(Z_y, L_y)$, the curve $s$ bounds a reducing annulus with an \emph{essential} curve in $\boundary_- Z_x$ or $\boundary_- Z_y$ then (3) holds and neither (1) nor (2) hold. If in either $(Z_x, L_x)$ or $(Z_y, L_y)$, the curve $s$ bounds a reducing annulus with an \emph{inessential} curve in $\boundary_- Z_x$ or $\boundary_- Z_y$, then both (1) and (3) hold or both (2) and (3) hold.
\end{lemma}
\begin{proof}
Let $E \subset \Sigma$ be a disk with boundary $s$ and suppose $D$ is a c-disk in $(Z_x, L_x)$ or $(Z_y, L_y)$ with boundary $s$. Then $D \cup E$ is a 2-sphere in $S^3$ or $B^3$ and, therefore, separates. Thus, if $E$ has an odd number of punctures, then $D$ is a cut disk and if $E$ has an even number of punctures $D$ is a compressing disk. Therefore, if $s$ bounds a c-disk in both $(Z_x, L_x)$ and $(Z_y, L_y)$ then exactly one of (1) or (2) holds. Consequently, without loss of generality, we may suppose that $s$ bounds a c-disk $D$ in $(Z_y, L_y)$ and a spanning annulus $P$ in $(Z_x, L_x)$ that is disjoint from $L$ with a curve $s' \subset \boundary_- Z_x$. Let $E' \subset \boundary_- Z_x$ be the disk with boundary $s'$. Then $E' \cup P \cup D$ is a sphere in $Z$ bounding a ball containing the properly embedded disk $E$. Since $(Z,L)$ is a spanning relative tangle, $E'$ and $D$ must both have one puncture or both have no punctures. In particular, $s'$ is inessential. After a small isotopy to make it properly embedded, $E' \cup P$ is a c-disk in $(Z_x, L_x)$, so either (1) or (2) holds. Hence, we can assume (3) holds. Then $s$ bounds vertical annuli $P_x$ and $P_y$ in $(Z_x, L_x)$ and $(Z_y, L_y)$ disjoint from $L$ with curves $s_x \subset \boundary_- Z_x$ and $s_y \subset \boundary_- Z_y$, respectively. Let $E_x \subset \boundary_- Z_x$ and $E_y \subset \boundary_- Z_y$ be the disks bounded by $s_x$ and $s_y$. If $s$ also bounds a c-disk $D$ in $(Z_x, L_x)$ (say) then $D \cup P_x \cup E_x$ is a sphere in $Z$ and we see that $s_x$ is inessential. Thus, if one of $s_x$ or $s_y$ is essential, neither (1) nor (2) holds, while if both are inessential then (1) or (2) also holds.
\end{proof}

\begin{corollary}[The Consistent Bounding Corollary]\label{consistent bounding cor}
Suppose that $(Z_1, L_1)$, $(Z_2, L_2)$, and $(Z_3, L_3)$ are all either trivial tangles in $B^3$ or strictly trivial tangles in $D^2 \times I$ such that for each $i \neq j$, $\boundary_+ Z_i = \boundary_+ Z_j = Z_i \cap Z_j$ and $(Z_i, L_i) \cup (Z_j, L_j)$ is either an unlink or a spanning trivial relative tangle with bridge sphere or disk $\Sigma = \boundary_+ Z_i = \boundary_+ Z_j$. Suppose that in $\Sigma$ there exists an essential curve $s$ such that in each of $(Z_1, L_1)$, $(Z_2, L_2)$, and $(Z_3, L_3)$, the curve $s$ bounds either a c-disk or spanning annulus. Then one of the following holds:
\begin{enumerate}
\item In each $(Z_i, L_i)$, the curve $s$ bounds a compressing disk.
\item In each $(Z_i, L_i)$, the curve $s$ bounds a cut disk.
\item In each $(Z_i, L_i)$, the curve $s$ bounds a spanning annulus with a curve of $\boundary_- Z_i$ 
\end{enumerate}
\end{corollary}
\begin{proof}
This is almost immediate from Lemma \ref{consistent bounding}. For $i, j = 1,2,3$ with $i \neq j$, use the notation $(ij.N)$ to indicate that conclusion ($N$) from Lemma \ref{consistent bounding} holds for $(Z_i, L_i)$ and $(Z_j, L_j)$. Suppose that  $(12.1)$ or $(12.2)$ holds. If $(13.1)$, $(13.2)$, or $(23.1)$, or $(23.2)$ hold, we are done by Lemma \ref{consistent bounding}. So suppose that $(13.3)$ and $(23.3)$ hold. This implies Conclusion (3). Suppose, therefore, that $(12.3)$ holds and $(12.1)$ and $(12.2)$ do not.  If $(13.3)$ or $(23.3)$ hold, then we again have Conclusion (3). Thus, (13.1) or (13.2) hold and also (23.1) or (23.2). Again, by Lemma \ref{consistent bounding}, we are done.
\end{proof}

In our quest to find spheres or annuli to surger along, we will make use of the following lemma. In the case when $\tau$ is the unknot, this is essentially due to Otal \cite{Otal}. In the case of relative tangles, after applying a trick, described below, Hayashi and Shimokawa \cite{HS1} handled the case when every component of $\tau$ is either an arc with both endpoints on $\boundary_+ Z$ or an arc with one endpoint on each of $\boundary_\pm Z$. The case when $Z \setminus \tau$ contains an essential sphere is also due to Hayashi-Shimokawa \cite[Theorem 1.4]{HS2}.

\begin{lemma}\label{reduction lem}
Suppose that $(Z, \tau)$ is either an $(S^3, \text{ unlink})$ pair or a spanning relative trivial tangle with $\tau \neq \nil$ and that $\Sigma$ is a bridge sphere or disk, respectively. Then the following hold:
\begin{enumerate}
\item If $Z\setminus \tau$ contains an essential sphere, then there exists an unpunctured reducing sphere for $\Sigma$.
\item If  one of the following holds:
\begin{itemize}
\item $\tau$ is the unknot and $\Sigma$ is a sphere such that $|\tau \cap \Sigma| = 4$, or 
\item $\tau$ is a single arc and $\Sigma$ is a disk such that $|\tau \cap \Sigma| = 3$, 
\end{itemize}
 then there exist compressing disks for the tangles on either side of $\Sigma$ with boundaries on $\Sigma$ intersecting exactly twice.
\item If one of the following holds:
\begin{itemize}
\item $\tau$ is the unknot, $\Sigma$ is a disk and $|\tau \cap \Sigma| \geq 4$, 
\item $\tau$ is the unknot, $\Sigma$ is a sphere and $|\tau \cap \Sigma| \geq 6$, or
\item $\tau$ is a collection of arcs, $\Sigma$ is a disk, and $\tau$ contains an arc intersecting $\Sigma$ in 3 or more points,
\item $\tau$ is a collection of arcs, $\Sigma$ is a disk, and $|\tau \cap \Sigma|\geq 4$,
\end{itemize}
then there exists a twice-punctured reducing sphere for $\Sigma$.
\item If $\tau$ is the union of three or more arcs, $\Sigma$ is a disk, and each arc of $\tau$  intersects $\Sigma$ in exactly one point, then there exists an even reducing annulus for $\Sigma$. 
\end{enumerate}
\end{lemma}
\begin{proof}

We defer to \cite{HS2} for the case when $Z \setminus \tau$ is reducible. For (2) and (3), we will show that $\Sigma$ is ``perturbed'' and that this implies produces the desired conclusion. Assume, therefore, that $Z \setminus \tau$ is irreducible.

A \defn{bridge disk} for a bridge arc $\tau_0 \subset \tau\setminus \Sigma$ with endpoints on $\Sigma$ is an embedded disk with interior disjoint from $\Sigma \cup \tau$ and whose boundary is the union of $\tau_0$ with an arc on $\Sigma$. We say that $\Sigma$ is \defn{perturbed} if there exist bridge disks (called a \defn{perturbing pair}) on opposite sides of $\Sigma$ whose arcs on $\Sigma$ are disjoint except for sharing a single endpoint. Observe that the boundary of a regular neighborhood of the union $S$ of two such bridge disks is a twice-punctured sphere in $Z$ intersecting $\Sigma$ in a single simple closed curve. Thus, if $|\Sigma \cap \tau| + |\boundary\Sigma| \geq 5$ and $\Sigma$ is perturbed, then there is a twice-punctured reducing sphere for $\Sigma$. Similarly, if $|\Sigma \cap \tau| + |\boundary\Sigma| \geq 4$ and $\Sigma$ is perturbed, then there exist compressing disks for the tangles on either side of $\Sigma$ with boundaries intersecting on $\Sigma$ exactly twice. 

If $Z = S^3$, since $\tau$ is the unlink and $Z \setminus \tau$ is irreducible, $\tau$ is the unknot. By \cite{Otal}, it is perturbed and so the lemma holds. Henceforth, assume that $Z = D^2 \times I$. Since $Z \setminus \tau$ is irreducible, $\tau$ contains arc components. Attach a copy of $D^2 \times I$ to $(\boundary D^2) \times I \subset \boundary Z$ so that the product structures match. This converts $Z$ into $\wihat{Z} = S^2 \times I$. In the newly attached $D^2 \times I$, choose a vertical arc and call it $\tau'$. Let $\wihat{\tau} = \tau \cup \tau'$. Notice that we can recover $(Z,\tau)$ from $(\wihat{Z},\wihat{\tau})$ by drilling out $\tau'$. Attach a disk of the form $D^2 \times \{t_0\}$ to $\boundary \Sigma$ to form $\wihat{\Sigma}$. Since $\wihat{\Sigma}$ intersects $\tau'$ only once, there is no perturbing pair for $\tau'$.  Conclusions (2) and (3) now follow from \cite{HS1}.

Finally, suppose that $\tau$ is the union of three or more arcs, each intersecting $\Sigma$ in a single point. From the definition of spanning relative trivial tangle, we see that $\tau$ can be properly isotoped, by an isotopy preserving $\Sigma$ so that the arcs $\tau$ are vertical in the product structure on $Z$. We may then find an essential curve $s \subset \Sigma$ bounding a disk in $\Sigma$ containing an even number of punctures. Such a curve bounds even spanning annuli to both sides and their union is an even reducing annulus. Reversing the isotopy, we find the desired reducing annulus in $(Z,\tau)$.
\end{proof}

\section{Connected sums and distant sums of bridge trisections}\label{sec:sums}

Given two properly embedded smooth surfaces $S_1$ and $S_2$ such that one of them is in $S^4$ and the other is in either $B^4$ or $S^4$, we can form either their \defn{connected sum} $S_1 \# S_2$ or their \defn{distant sum} $S_1 \sqcup S_2$, which will be a surface in either $B^4$ or $S^4$. On the ambient 4-manifolds, both the connected sum and the distant sum function as connected sums; the difference is that for the connected sum the summing points are chosen to lie on $S_1$ and $S_2$ and for the distant sum the summing points are chosen to be disjoint from $S_1$ and $S_2$. Whether we perform a connected sum or distant sum, if $\mc{T}_1$ and $\mc{T}_2$ are bridge trisections for $S_1$ and $S_2$ with trisection surfaces $\Sigma_1$ and $\Sigma_2$ respectively, and if the summing points are chosen to lie on $\Sigma_1$ and $\Sigma_2$, then the connected sum $\Sigma= \Sigma_1 \# \Sigma_2$ is a trisection surface for a bridge trisection $\mc{T}$ of $S = S_1 \# S_2$ or $S = S_1 \sqcup S_2$. When we perform the connected sum or distant sum, we also say that the trisection $\mc{T}$ is a connected sum or distant sum, respectively. In either case, the trisection surface for $\mc{T}$ is the connected sum of the trisection surfaces for $\mc{T}_1$ and $\mc{T}_2$. See \cite[Section 2.2]{MZ1} for more details on connected sum. 

Observe that in a trisection surface $\Sigma$ for a trisection $\mc{T}$ that is a connected sum or distant sum, there is a simple closed curve $s \subset \Sigma$ such that the summing 3-sphere $P$ in the ambient 4-manifold intersects $\Sigma$ in $s$.  In the case of the connected sum, $P$ also intersects $S_1\# S_2$ in a single simple closed curve. This 3-sphere $P$ has the property that performing surgery along $P$ allows us to recover $S_1$ and $\mc{T}_1$, as well as $S_2$ and $\mc{T}_2$. In each trivial tangle in the spine of $\mc{T}$, the curve $s$ bounds either a disk disjoint from $S_1 \sqcup S_2$ (in the case of a distant sum) or a disk intersecting $S_1 \# S_2$ in a single point (in the case when the sum is a connected sum). If the sum is a distant sum, then the curve $s$ is  essential in $\Sigma$ as neither $S_1$ nor $S_2$ is empty. If the sum is a connected sum, then $s$ is essential in $\Sigma$ if and only if neither $\mc{T}_1$ nor $\mc{T}_2$ have bridge number 1/2 or 1. Moreover, we will show that if the loop $s$ is inessential, then $S_1$ or $S_2$ is an unknotted 2-sphere  in $S^4$ or $\boundary$-parallel disk in $B^4$. Thus, if $s \subset \Sigma$ is an essential simple closed curve, we say that the trisection $\mc{T}$ and trisection surface $\Sigma$ are a \defn{nontrivial connected sum} or \defn{distant sum}. Modelling our terminology on that of Heegaard splitting theory, we say that a bridge trisection $\mc{T}$ with trisection surface $\Sigma$ and spine $\mathcal{S}=(Z_{12},T_{12})\cup(Z_{23},T_{23})\cup(Z_{31},T_{31}) \cup (V, V \cap \boundary S)$ is \defn{reducible} if there exists an essential simple closed curve $s\subset \Sigma$ such that either for each choice of distinct $i,j \in \{1,2,3\}$ the curve $s$ bounds a compressing disk in $(Z_{ij}, T_{ij})$ or for each choice of distinct $i,j \in \{1,2,3\}$, the curve $s$ bounds a cut disk in $(Z_{ij}, T_{ij})$.

\begin{lemma}\label{redcriterion}
A bridge trisection $\mc{T}$ of a surface $S$ in either $B^4$ or $S^4$ is a nontrivial connected sum or distant sum if and only if it is reducible.
\end{lemma}
\begin{proof}
The ``only if'' direction was addressed in the previous paragraph. To prove the ``if'' direction, suppose that $\mc{T}$ is reducible. Let $s \subset \Sigma$ be an essential curve in the trisection surface for $\mc{T}$ such that $s$ bounds either a compressing disk in each tangle $(Z_{ij}, T_{ij})$ forming the spine for $\mc{T}$ or $s$ bounds a cut disk in each tangle $(Z_{ij}, T_{ij})$ forming the spine for $\mc{T}$. In either case, let $D_{ij} \subset Z_{ij}$ be the c-disk.  In both the punctured and unpunctured cases, cutting the trivial tangle $(Z_{ij},T_{ij})$ along  $D_{ij}$ produces two trivial tangles, and surgering along the union $D_{ij} \cup D_{jk}$ in the 3-sphere or 3-ball $Z_{ij}\cup Z_{jk}$ decomposes the unlink or relative spanning trivial tangle $T_{j} = T_{ij} \cup T_{jk}$ into the disjoint union of links or tangles. Since $T_j$ was the unlink or a relative spanning trivial tangle, these links or tangles are unlinks or relative spanning trivial tangles. Thus, we have spines for bridge trisections $\mc{T}_1$ and $\mc{T}_2$ of surfaces $S_1$ and $S_2$, with one surface being in $S^4$ and the other being in $B^4$ or $S^4$.  Performing the distant sum or connected sum of these trisections produces a trisection $\mc{T}'$ with the same spine as that of $\mc{T}$. By Lemma \ref{lemma:spine}, this means that the original trisection is the distant or connected sum of these two other trisections.
\end{proof}

In a similar spirit, we have the following characterization of trisections of surfaces in $B^4$ that are either closed or  whose boundary is disjoint from or a 1-braid with respect to the boundary of the trisection surface.

\begin{lemma}\label{one arc}
Suppose that $\mc{T}$ is a relative trisection of $S \subset B^4$ with spine $\mathcal{S}=(Z_{12},T_{12})\cup(Z_{23},T_{23})\cup(Z_{31},T_{31}) \cup (V, V \cap \boundary S)$. If $\boundary S$ is a 0-braid or 1-braid in $V$ (equivalently, if each $T_{ij}$ contains at most one vertical arc), then $\mc{T}$ is the distant sum or connected sum of a trisection for a surface $S' \subset S^4$ of bridge number $|\Sigma \cap S|/2$ with either the relative bridge trisection of the empty surface in $B^4$ or a bridge number 1/2 relative bridge trisection of a $\boundary$-parallel disc in $B^4$.
\end{lemma}
\begin{proof}
Let $s \subset \Sigma$ be a $\boundary$-parallel simple closed curve. Since each $T_{ij}$ contains at most one  vertical arc, the curve $s$ bounds a zero-punctured or once-punctured disk in each $(Z_{ij}, T_{ij})$. (It must be the same type of disk in each.) As in the proof of Lemma \ref{redcriterion}, we can surger the spine for $\mc{T}$ along those zero-punctured or once-punctured disks and this extends to a surgery on $S$ and $\Sigma$ in $B^4$. The result is a trisection $\mc{T}_1$ of a surface $S_1 \subset B^4$ and a trisection $\mc{T}_2$ of a surface $S_2 \subset S^4$. The trisection $\mc{T}_1$ has bridge number 0 or 1/2, as $s$ was $\boundary$-parallel in $\Sigma$. The trisection $\mc{T}_2$ has bridge number equal to that of $\mc{T}$. By Lemma \ref{only arcs}, the surface $S_1$ is either empty or a $\boundary$-parallel disk.
\end{proof}


  \section{The pants complex and efficient defining pairs}
  \label{pants}
   
Suppose that $\Sigma$ is a compact surface with punctures. A \emph{pants decomposition} of $\Sigma$ is a collection of pairwise disjoint essential curves (up to isotopy) cutting $\Sigma$ into pairs of pants. The cases that are of most interest to us are when $\Sigma$ is an admissible punctured sphere or disk. If $\Sigma$ is a sphere with $2b \geq 4$ punctures, then each pants decomposition of $\Sigma$ has $2b - 3$ curves. If $\Sigma$ is a disk with $2b \geq 3$ punctures, then each pants decomposition of $\Sigma$ has $2b - 2$ curves. Define $\mc{P}(\Sigma)$, the \emph{pants complex} of $\Sigma$, as follows. Each pants decomposition of $\Sigma$ is a vertex of $\mc{P}(\Sigma)$. Two vertices are connected by an edge if the two corresponding pants decompositions have all but one (isotopy class of) curve in common and the two curves where they differ (have representatives that) intersect minimally in exactly two points.  The distance $d(x,y)$ between two collections of  vertices $x$ and $y$ in $\mc{P}(\Sigma)$ is the minimum number of edges in a path in $\mc{P}(\Sigma)$ between a vertex of $x$ and a vertex of $y$. The pants complex for admissible surfaces is connected and has infinite diameter, using the metric $d$ \cite{HatcherThurston}. We say that a curve $s \subset \Sigma$ is a \defn{common curve} for a given path in $\mc{P}(\Sigma)$ if, for each pants decomposition on the path, $s$ is isotopic to one of the curves in the pants decomposition.

We will be using curves in pants decompositions to find reducing spheres and annuli. The combinatorics in Lemma \ref{DisjointSpheres} are the reason for insisting that our reducing annuli be even.

If $\Sigma'$ is a disk with punctures (not necessarily admissible), a collection of pairwise disjoint essential simple closed curves is a \defn{weak pants decomposition} of $\Sigma'$ if it is either empty and $\Sigma'$ has two or fewer punctures, or if it cuts $\Sigma'$ into pairs of pants and annuli, such that at most one of the pairs of pants is a once-punctured annulus and that annulus, if it exists, has $\boundary \Sigma'$ as one of its boundary components. For both trivial tangles and strictly trivial tangles $(B, \kappa)$ with $\Sigma = \boundary_+ B$, we define a certain subset of the vertices of $\mc{P}(\Sigma)$ to be the \defn{disk set} $\mc{D}_\kappa$. The case when $(B, \kappa)$ is a trivial tangle is the simplest.

\begin{definition}
 Suppose that $(B, \kappa)$ is a trivial tangle with $\Sigma = \boundary B$ admissible.  A vertex $x \in \mc{P}(\Sigma)$ lies in $\mc{D}_\kappa$ if and only if there is a collection of c-disks $D \subset B$ such that $x = \boundary D$. Now suppose that $(B,\kappa)$ is a strictly trivial relative tangle with $\Sigma = \boundary_+ B$ admissible. A vertex $x \in \mc{P}(\Sigma)$ lies in $\mc{D}_\kappa$ if and only if there is a collection of properly embedded disks and annuli $D \subset B$, transverse to $\kappa$ such that:
 \begin{enumerate}
 \item Each disk component of $D$ is a c-disk for $(B, \kappa)$
 \item Each annulus component of $D$ is an even spanning annulus for $(B,\kappa)$.
 \item $\boundary D \cap \Sigma = x$
 \item $\boundary D \cap \boundary_- B$ is a weak pants decomposition of $\Sigma$.
\end{enumerate}
 \end{definition}

It is a well-known fact (and easy to prove) that if $\Sigma = \boundary B$ is an admissible sphere for the trivial tangle $(B,\kappa)$, then $\mc{D}_{\kappa}$ is nonempty. We need the corresponding result for bridge disks. See Figure \ref{fig:pants} for an example.

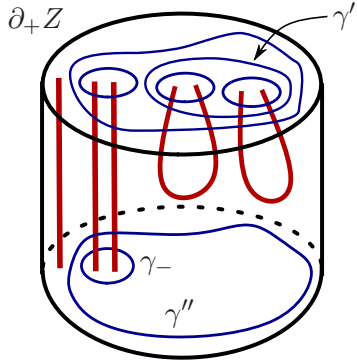
\begin{figure}[ht!]
 \labellist
\small\hair 2pt
 \pinlabel {$\boundary_+ Z$} [br] at 35 404
\pinlabel {$\gamma'$} [l] at 361 432
\pinlabel {$\gamma_-$} [l] at 119 118
\pinlabel {$\gamma''$} [b] at 176 34
\endlabellist
 \centering
 \includegraphics[scale=0.3]{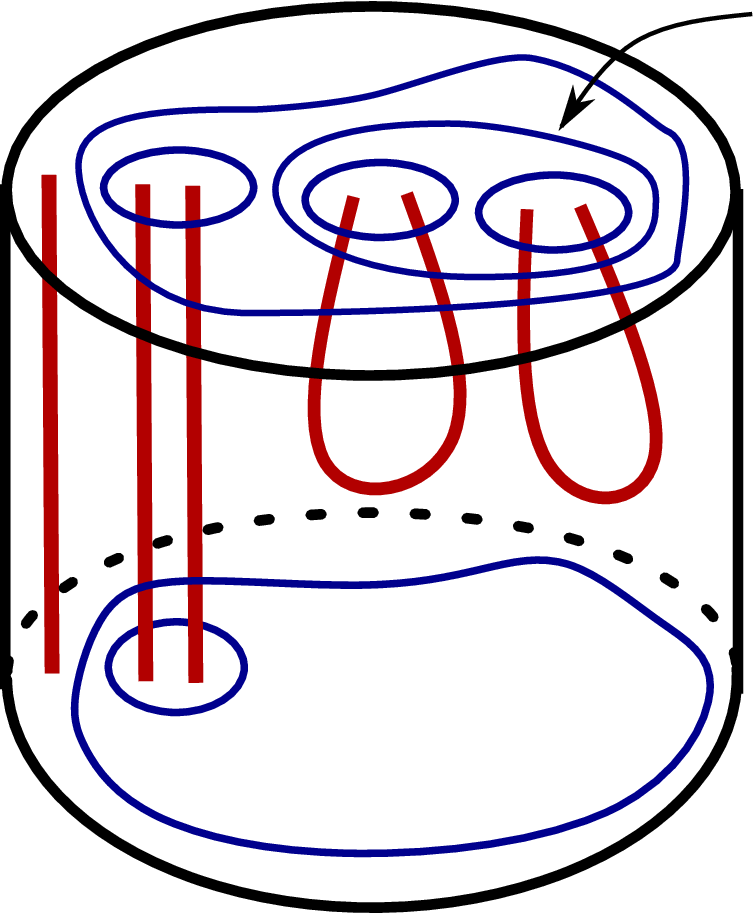}
 \caption{The thin blue curves at the top of $Z = D^2 \times I$ form a pants decomposition in the disk set of the strictly trivial relative tangle constructed according to our recipe. We have labelled the curves $\gamma_-$ and $\gamma''$ used in the construction in the proof of Lemma \ref{lem:nonemptydisksets}.}
 \label{fig:pants}
 \end{figure}

\begin{lemma}\label{lem:nonemptydisksets}
Suppose that $(B,\kappa)$ is a strictly trivial relative tangle such that $|\boundary_+ B \cap \kappa| \geq 3$. Then $\mc{D}_\kappa \neq \nil$.
\end{lemma}
\begin{proof}
Let $\alpha \subset \boundary_+ B =\Sigma$ be the trace arc. By choosing $\alpha$ carefully we can guarantee that as we traverse $\alpha$ (in some direction)  we encounter the endpoints of all vertical arcs prior to encountering the endpoints of all bridge arcs and also that there is no nesting of the bridge arcs of $\kappa$ in the trace disk $\alpha \times I$. The bridge arcs cut off bridge disks from $\alpha \times I$. And we can pair up the adjacent vertical arcs so that, in pairs, they form two edges of the boundary of a rectangle $\alpha' \times I$ where $\alpha'$ is a subarc of $\alpha$ and $int(\alpha' \times I)\cap \kappa=\nil$. We call such a rectangle a \defn{parallellism}. If the number of vertical arcs is odd, there will be one vertical arc left over and we arrange for it to be the vertical arc whose endpoint is the first we encounter as we traverse $\alpha$. 

Let $\delta$ be the union of the bridge disks and parallelisms.  Let $\alpha' \subset \alpha$ be the subarc containing exactly those punctures of $\kappa \cap \boundary_+ B$ belonging to bridge arcs. Let $\alpha_+  = \delta \cap \boundary_+ B$. The boundary of a regular neighborhood of $\alpha_+$ is a collection of simple closed curves $\gamma_+$ in $\boundary_+ B$, each bounding a twice-punctured disc in $\boundary_+ B$ and each bounding an unpunctured disc in $(B, \kappa)$. 

If there are no vertical arcs, set $\gamma' = \boundary (\boundary_+ B)$ and $E = \boundary_+ B$. Otherwise, let $\gamma'$ be the boundary of a regular neighborhood of $\alpha'$ and isotope it as necessary so that it bounds a disc $E$ in $\boundary_+ B$ containing $\gamma_+$ in its interior. Observe that $\gamma'$ bounds an unpunctured disc $D$ in $(B, \kappa)$.

Suppose, for the moment, that there are at least three vertical arcs. After boundary-reducing $(B, \kappa)$ along $D$, it becomes $(B', \kappa') = (D^2 \times I, \text{ points } \times I)$. Using that product structure, the remnant of the disc $D$ in $\boundary_+ B'$ projects vertically to an unpunctured disc $D' \subset \boundary_- B' = \boundary_- B$. The parallelisms in $\alpha \times I$ survive to parallelisms between the arcs of $\kappa'$. Let $\delta_-$ be the arcs that are the intersections between these parallelisms and $\boundary_- B'$. Let $\gamma_-$ be the boundary of a regular neighborhood of $\delta_-$. Then each curve of $\gamma_- $ bounds a twice-punctured disc in $\boundary_- B'$ and additional curves can be added to $\gamma_-$ so that it is a pants decomposition of $\boundary_- B'$ and is disjoint from $D'$. Extend $\gamma_-$ vertically through the product structure on $(B', \kappa')$ to arrive at curves $\wihat{\gamma}_-$ in $\boundary_+ B$ disjoint from $E$. The curves $\wihat{\gamma}_-$ bound even spanning annuli. Additional curves can then be added to $\wihat{\gamma}_- \cup \gamma' \cup \gamma_+$ to turn it into an element of $\mc{D}_\kappa$. The curves that are added either lie in $E$ and bound unpunctured discs in $(B, \kappa)$ or are external to $E$ and bound even spanning annuli.  If, on the other hand, there are two or fewer vertical arcs, the construction of an element in $\mc{D}_\kappa$ is easier as the empty set of curves is a weak pants decomposition of $\boundary_- B$. 
\end{proof}

\begin{definition}
Suppose that $(Z,\tau)$ is either an unlink or a relative spanning trivial tangle with bridge surface $\Sigma$ dividing $(Z,\tau)$ into trivial tangles or strictly trivial relative tangles $(B,\kappa)$ and $(B', \lambda)$. A pair of pants decompositions $x \in \mc{D}_\kappa$ and $y \in \mc{D}_\lambda$ is said to be an \defn{efficient defining pair} if $d(x,y) = d(\mc{D}_\kappa, \mc{D}_\lambda)$. (That is, if $x$ and $y$ are vertices in $\mc{D}_\kappa$ and $\mc{D}_\lambda$ whose distance in the pants complex of $\Sigma$ is minimal.)
\end{definition} 
 
We can now define the Kirby-Thompson invariant of a bridge trisection. See Figure \ref{ldefn2} for a schematic representation of the efficient defining pairs for a trisection.

\begin{definition}
Suppose that $S \subset B^4$ or $S \subset S^4$ is a properly embedded surface with bridge trisection $\mc{T}$ having trisection surface $\Sigma$ and spine $(Z_{12}, T_{12}) \cup_\Sigma (Z_{23}, T_{23}) \cup_\Sigma (Z_{13}, T_{13})\cup (V,\partial S \cap V)$. For $\{i,j,k\} = \{1,2,3\}$, let $(p^j_{ij}, p^j_{jk})$ be an efficient defining pair for $(Z_{ij}, T_{ij}) \cup_\Sigma (Z_{jk}, T_{jk})$. If $\Sigma$ is not admissible, define $\L(\mc{T}) = 0$. Otherwise, define $\L(\mc{T})$ to be the minimum of
\[
d(p^1_{12}, p^2_{12}) + d(p^1_{31}, p^3_{31}) + d(p^2_{23}, p^3_{23})
\]
over all such choices of efficient defining pairs. Define $\L(S)$ to be the minimum of $\L(\mc{T})$ over all trisections $\mc{T}$ of $S$ with $b(\mc{T}) = b(S)$.
\end{definition}
     
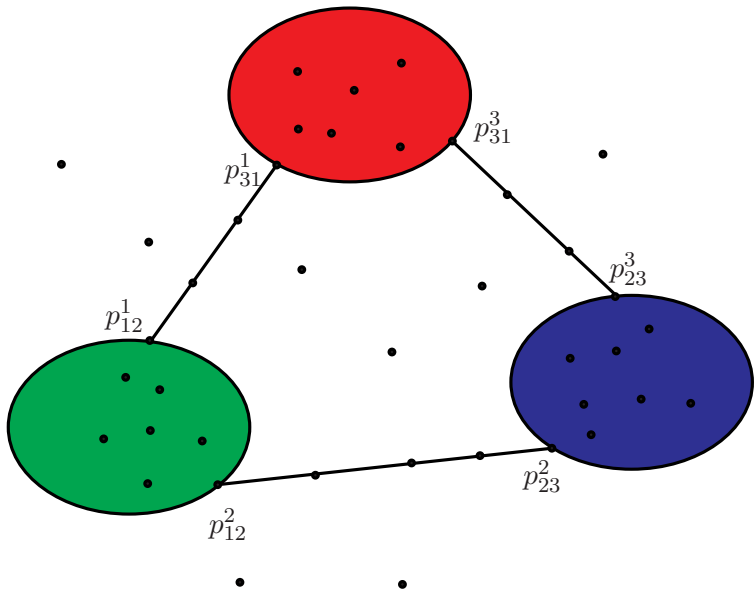
\begin{figure}[ht!]
\labellist
\small\hair 2pt
\pinlabel $p_{12}^1$ [br] at 94 190
\pinlabel $p_{12}^2$ [br] at 160 57
\pinlabel $p_{23}^2$ [br] at 360 88
\pinlabel $p_{23}^3$ [br] at 415 220
\pinlabel $p_{31}^3$ [br] at 329 310
\pinlabel $p_{31}^1$ [br] at 170 283
\endlabellist
\centering
\includegraphics[width=4in]{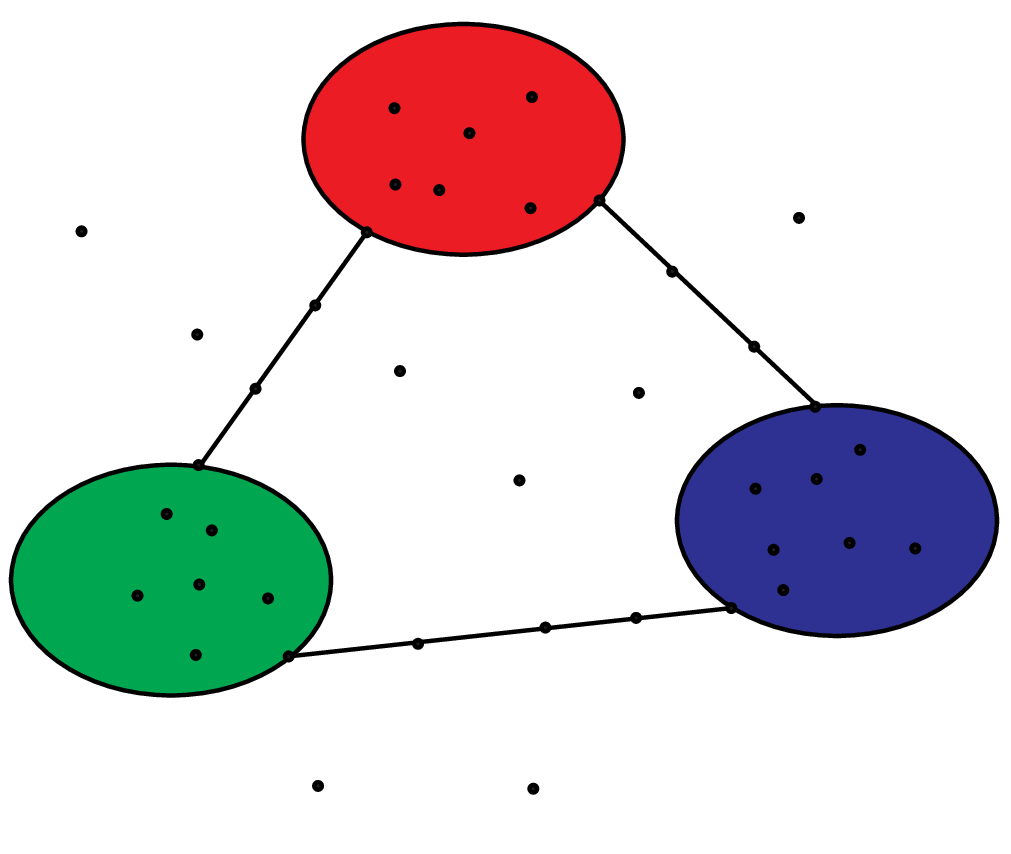}
\caption{Defining $\L(\mc{T})$ via efficient defining pairs. The dots represent pants decompositions of $\Sigma$ and the dark ellipses represent the disk sets. The line joining $p^i_{jk}$ to $p^i_{ik}$ represents a geodesic path in the pants complex.}
\label{ldefn2}
\end{figure}

\begin{remark}
Our definition of $\L$ is not an exact parallel of that of Kirby and Thompson \cite{KirbyThompson} in three main regards. The first is that we use the pants complex, rather than the cut complex, which is empty for punctured spheres and disks. Secondly, we calculate the distance between $p^i_{ij}$ and $p^j_{ij}$ in the whole pants complex, rather than in the disk set. This is not a serious point, but it allows us to avoid discussing the geometry of the disk set. Finally, in defining $\L(S)$ we minimize $\mc{L}(\mc{T})$ only over those trisections achieving the bridge number of $S$, rather than over all trisections of $S$. In the context of bridge trisections, this may be the more sensible definition for the following reason. Gay and Kirby show that any two trisections of a 4-manifold become equivalent after stabilizing each some number of times. Kirby and Thompson's invariant does not increase under stabilization; thus, Kirby and Thompson's definition of their invariant as a minimum, is equivalent to taking its limit under trisection stabilization. Bridge trisections, on the other hand, behave differently. Meier and Zupan \cite{MZ1} show that any two trisections of a surface in $S^4$ become equivalent after some number of perturbations and unperturbations, applied to each. (This was extended to higher genus bridge trisections in \cite{HKM}.) Consequently, minimizing $\mc{L}$ over all bridge trisections may no longer be equivalent to taking a limit.
\end{remark}

The remainder of this section is taken up with developing the properties of efficient defining pairs and the geodesics in $\mc{P}(\Sigma)$ between them.  Our analysis draws heavily from that of Zupan \cite[Lemmas 4.1 and 4.2]{Zupan}. We give the complete proofs since our setting is somewhat different and since Zupan does not address the relative case.

\subsection{The distance between an efficient defining pair}
In this subsection, we establish the distance in the pants complex between $x$ and $y$ forming an efficient defining pair for some unlink or spanning relative trivial tangle.

\begin{lemma}  \label{lemma:mindistance}
Let $(S^3, L)$ be a link with bridge sphere $\Sigma$ such that $b = |\Sigma \cap L|/2 \geq 2$ and having $c \geq 1$ components. Let  $(B,\kappa)$ and $(B',\lambda)$ be the tangles on either side of $\Sigma$.   For all $x \in \mc{D}_\kappa$ and $y \in \mc{D}_\lambda$, we have $d(x,y) \geq b - c$. Furthermore, if $L$ is an unlink, then for any efficient defining pair $x \in \mc{D}_\kappa$ and $y \in \mc{D}_\lambda$, equality holds. Additionally, if $b\geq 3$ and $d(x,y) = b - c$, then for any geodesic in $\mc{P}(\Sigma)$ between the efficient pair there is a common curve.
\end{lemma}

\begin{proof}
We first show that if $L$ is the unlink of $c$ components, then there exist $x \in \mc{D}_\kappa$ and $y \in \mc{D}_\lambda$ with $d(x,y) = b- c$. If $b = 2$ and $c = 2$, by part (1) of Lemma \ref{reduction lem}, there is an unpunctured sphere intersecting $\Sigma$ in a single essential simple closed curve $x$. Observe that setting $x = y$, we have $x \in \mc{D}_\kappa$, $y \in \mc{D}_\lambda$ and $0 = d(x,y) = b-c$. If $b = 2$ and $c = 1$, then by part (2) of Lemma \ref{reduction lem}, there exist compressing disks $D_\kappa$ in $(B,\kappa)$ and $D_\lambda$ in $(B', \lambda)$ with $x = \boundary D_x$ and $y = \boundary D_y$ intersecting twice. Thus, $x \in \mc{D}_x$ and $y \in \mc{D}_y$ and we have $1 = d(x,y) = b - c$. Assume, therefore, that $b \geq 3$ and that for any unlink $L'$ of $c'$ components having a bridge sphere with $4 \leq 2b' < 2b$ points of intersection with $L'$, we have pants decompositions for the trivial tangles on either side of the bridge sphere of distance at most $b' - c'$ from each other and lying in the corresponding disk sets. By parts (1) and (3) of Lemma \ref{reduction lem}, there exists an unpunctured or twice-punctured reducing sphere $S$ for $\Sigma$. Surger $(S^3, L)$ and $\Sigma$ along $S$ to obtain links $(S^3, L_1)$ and $(S^3, L_2)$ with bridge spheres $\Sigma_1$ and $\Sigma_2$. Let $b_1$ and $b_2$ be the corresponding bridge numbers and $c_1 = |L_1|$ and $c_2=|L_2|$. Recall that $b_1 + b_2 - (c_1 + c_2) = b-c$ and both $b_1$ and $b_2$ are at most $b-1$, since $S \cap \Sigma$ is essential in $\Sigma$.  Notice also that if we transversally orient $\Sigma$, then $\Sigma_1$ and $\Sigma_2$ also inherit transverse orientations. For $i = 1,2$, if $b_i \geq 2$, let $x_i$ and $y_i$ be pants decompositions for $\Sigma_i$ with $x_i$ lying in the disk set for the tangle below $\Sigma_i$ and $y_i$ lying in the disk set for the tangle above $\Sigma_i$. By our inductive hypothesis, we may choose $x_i$ and $y_i$ so that $d(x_i, y_i) \leq b_i - c_i$. If $b_i = 1$, set $x_i = y_i = \nil$. The curves $x = x_1 \cup x_2 \cup (\Sigma \cap S)$ and $y = y_1 \cup y_2 \cup (\Sigma \cap S)$ are pants decompositions for $\Sigma$ lying in $\mc{D}_\kappa$ and $\mc{D}_\lambda$ respectively. Geodesic paths in $\mc{P}(\Sigma_1)$ and $\mc{P}(\Sigma_2)$ from $x_1$ to $y_1$ and from $x_2$ to $y_2$ can be concatenated to produce a path from $x$ to $y$ in $\mc{P}(\Sigma)$ of length at most $b-c$, as desired.

Now suppose that $L$ is any link. Since $(B, \kappa)$ and $(B', \lambda)$ are trivial tangles, each with at least two arcs, the sets $\mc{D}_\kappa$ and $\mc{D}_\lambda$ are non-empty. Thus, $d(\mc{D}_\kappa,\mc{D}_\lambda)$ is well-defined. Choose $x \in \mc{D}_\kappa$ and $y \in \mc{D}_\lambda$ so that $d(x,y) = d(\mc{D}_\kappa,\mc{D}_\lambda)$. We will show that $d(x,y) \geq b- c$ by induction on $b \geq 2$. 

Consider first the case when $b=2$, so $x$ and $y$ each correspond to a single curve on the 4-punctured sphere $\Sigma$.  If $x$ and $y$ correspond to the same curve, this curve must bound compressing disks on both sides, so $L$ is a two component unlink and $0 = d(x,y) = b - c$, as desired. If the curve is not the same for both vertices, then $d(x,y)\geq 1$.  Since $b=2$ and $c\geq 1$, we have $d(x,y) \geq b - c$. 
   
Suppose the result is true for all bridge spheres with bridge number less than $b$, and $b$ is at least 3. Each vertex of $\mc{P}(\Sigma)$ corresponds to $2b-3$ curves and $b\geq c$. Choose a geodesic path in $\mc{P}(\Sigma)$ from $x$ to $y$ and let $\mathcal{C} \subset \Sigma$ be the collection of common curves for the geodesic. Adjacent vertices in $\mc{P}(\Sigma)$ differ by a single curve. Thus, if $\mathcal{C}=\nil$, we have $d(x,y) \geq 2b - 3 \geq b - c$, as desired. 
Suppose, therefore, that $\mc{C} \neq \nil$. 
   
Let $\gamma \in \mathcal{C}$. Since $\gamma$ is a component of both $x$ and $y$, it bounds a c-disk in each of $(B, \kappa)$ and $(B', \lambda)$. Since every sphere in $S^3$ separates, these are both compressing disks or both cut disks. The union of these disks is either an unpunctured reducing sphere or a twice-punctured reducing sphere $S$  for $\Sigma$. Surger $(S^3, L)$ and $\Sigma$ along $S$, to obtain links $(S^3, L_1)$ and $(S^3, L_2)$ and bridge spheres $\Sigma_1$ and $\Sigma_2$. Let $b_1$ and $b_2$ be the corresponding bridge numbers and $c_1 = |L_1|$ and $c_2 = |L_2|$. Recall that $b_1, b_2 < b$ and $(b_1 + b_2) - (c_1 + c_2) = b-c$. Let $x_i, y_i$ be the restrictions of $x$ and $y$ to $\Sigma_i$, and observe that either they are empty or they are pants decompositions lying in the disk sets for the tangles above and below $\Sigma_i$. If they are nonempty, since $\gamma \subset \mc{C}$, the geodesic in $\mc{P}(\Sigma)$ from $x$ to $y$, restricts to a geodesic in $P(\Sigma_i)$ from $x_i$ to $y_i$.  Let $D_i$ be the length of this geodesic, or 0 if $x_i$ and $y_i$ are empty.  Observe that $d(x,y) = D_1 + D_2$. Consequently, by our inductive hypothesis, $d(x,y) \geq (b_1 - c_1) + (b_2 - c_2) = b-c$, as desired.
\end{proof}


  We now turn to the relative case. For a bridge disk $\Sigma$ we let $b$  be half the number of punctures and for a tangle $(Z,\tau)$ we let $c$ be the number of closed components and $v$ the number of arc components of $\tau$.  Note that $\beta = b - (c+v/2)$ is a non-negative integer since each closed component of $\tau$ contributes a positive even number of punctures to $\Sigma$ and each arc component of $\tau$ contributes an odd number of punctures to $\Sigma$. Recall that when $b \geq 3/2$, by Lemma \ref{lem:nonemptydisksets}, the disk sets on either side of $\Sigma$ are nonempty.

    \begin{lemma} \label{lemma:RelativeMindistance}
 Let $(Z,\tau)$ be a spanning relative tangle having bridge disk $\Sigma$ such that $b \geq 3/2$. Let $(B, \kappa)$ and $(B', \lambda)$ be the tangles on either side of $\Sigma$. For all $x \in \mc{D}_\kappa$ and $y \in \mc{D}_\lambda$, $d(x,y) \geq \beta$. Furthermore, if $(Z,\tau)$ is trivial then for any efficient pair $x \in \mc{D}_\kappa$ and $y \in \mc{D}_\lambda$ equality holds. Additionally, if $b\geq 2$ and $d(x, y) = \beta$, then for any geodesic in $\mc{P}(\Sigma)$ between the efficient pair there is a common curve.
 \end{lemma}
   \begin{proof}

This proof is much like that of Lemma \ref{lemma:mindistance}, and we refer to that proof for several steps. In particular, if each of $\kappa$ and $\lambda$ either contains no vertical arc or a single vertical arc, then as in the proof of Lemma \ref{one arc}, we may decompose along a twice-punctured sphere $P \subset (Z, \tau)$ intersecting $\Sigma$ in a curve parallel to $\boundary \Sigma$. The result will be a bridge disk $\Sigma_1$ of bridge number  0 or 1/2 for a (possibly empty) tangle in $D^2 \times I$, together with a bridge sphere $\Sigma_2$ of bridge number $b$ for a link in $S^3$. In this case, the analysis of Lemma \ref{lemma:mindistance} applies directly to $\Sigma_2$ and our lemma can be easily derived from that. Suppose, therefore that at least one of $\kappa$ and $\lambda$ contains at least two vertical arcs. 

Suppose first that $(Z,\tau)$ is trivial. We induct on the half integer $b \geq 3/2$. Suppose $b = 3/2$. By our remarks above, at least one of $\kappa$ or $\lambda$ contains at least two vertical arcs, and is, therefore, the union of three vertical arcs. If the other contains a bridge arc, then we contradict the fact that $(Z, \tau)$ is spanning. So both $\kappa$ and $\lambda$ are the union of three vertical arcs. Each curve in $\Sigma$ bounding a twice-punctured disk in $\Sigma$, lies in both $\mc{D}_\kappa$ and $\mc{D}_\lambda$. In particular $\mc{D}_\kappa = \mc{D}_\lambda = \mc{P}(\Sigma)$. Consequently, $d(\mc{D}_\kappa, \mc{D}_\lambda) = 0 = \beta$ and the result holds. 

Suppose, therefore, that $b \geq 2$ and the result holds for all bridge disks of bridge number at least 3/2 and less than $b$.  By Lemma \ref{reduction lem},  $(Z,\tau)$ contains an unpunctured reducing sphere, twice-punctured reducing sphere, or even reducing annulus. In the proof of Lemma \ref{lemma:mindistance}, we handled the reducing spheres and the analysis here is essentially the same, so we do not repeat it. Consider, therefore, the case when there is an even reducing annulus $A$. Let $(Z_1, \tau_1)$ and $(Z_2, \tau_2)$ be the relative tangles resulting from surgery along $A$, and let $\Sigma_1$ and $\Sigma_2$ be the bridge disks. Let $b_i = |\Sigma_i \cap \tau_i|$ and let $c_i$ be the number of closed components of $\tau_i$ and $v_i$ the number of arc components. Observe that $b_1 + b_2 = b$, $c_1 + c_2 = c$, and $v_1 + v_2 = v$. If $b_i \geq 3/2$, let $x_i$ and $y_i$ be the pants decompositions of $\Sigma_i$ provided by the induction hypothesis. If $b_i < 3/2$, set $x_i = y_i = \nil$. Then, as in the proof of Lemma \ref{lemma:mindistance}, $x_1 \cup x_2 \cup (S \cap \Sigma)$ and $y_1 \cup y_2 \cup (S \cap \Sigma)$ are the desired pants decompositions with $d(x,y) \leq \beta$.

Suppose now that $(Z,\tau)$ is any spanning relative tangle with $b \geq 3/2$. Let $x \in \mc{D}_x$ and $y \in \mc{D}_y$.  We again induct on $b$.  As before the case when $b = 3/2$ can be handled by reducing to Lemma \ref{lemma:mindistance}. So assume $b \geq 2$.  We will show that $d(x,y) \geq \beta$ and that if equality holds then there is a common curve on each geodesic from $x$ to $y$.

Choose a geodesic in $\mc{P}(\Sigma)$ from $x$ to $y$ and let $\mc{C}$ be the set of common curves for the geodesic. Recall that every pants decomposition of $\Sigma$ contains $2b-2$ curves. Note that $b \geq c + v/2$. Equality holds if and only if each closed component intersects $\Sigma$ exactly twice and each arc component intersects exactly once. Thus, if $b \geq 2$, then $b + (c + v/2) > 2$, and so $2b - 2 > \beta$. If $\mc{C} = \nil$, then $d(x,y) \geq |x| = 2b - 2$ since every curve of $x$ must be moved while traversing the geodesic.  In which case, $d(x,y) > \beta$.

Suppose, therefore, that there is a curve $\gamma \in \mc{C}$. As $\gamma \subset x \cap y$, it bounds a c-disk or even spanning annulus to each side of $\Sigma$. By Lemma \ref{consistent bounding}, it bounds the same type of disk or annulus to both sides. If $\gamma$ bounds compressing disks or cut disks for both $(B_x, \alpha_x)$ and $(B_y, \alpha_y)$, their union is an unpunctured or twice-punctured reducing sphere for $\Sigma$, and we proceed as in Lemma \ref{lemma:mindistance}. Suppose that $\gamma$ bounds an even spanning annulus $A_x$ in $B_x$  and an even spanning annulus $A_y$ in $B_y$. Recall that as $x \in \mc{D}_x$, there are curves $x' \subset \boundary_- B_x$ giving a weak pants decomposition of $\boundary_- B_x$ so that there are annuli joining $x'$ to a subset of $x$. All other curves of $x$ bound compressing or cut-disks in $(B_x, \alpha_x)$. By the definition of disk set, each curve of $x'$ bounds a disk in $\boundary_- B_x$ containing an even number of punctures. Similar statements hold for the curves of $y$. Thus, $A = A_x \cup A_y$ is an even reducing annulus. The argument now mimics that of the reducing sphere case to achieve the desired conclusion. 
\end{proof}  

\subsection{Common cut disks}
In this section, we study the properties of geodesics between the terms of an efficient defining pair. The goal is to show that if the distance is positive, then all curves in either pants decomposition of the efficient defining pair that bound cut disks are common curves.
   
 \begin{lemma}   \label{DisjointSpheres}
Suppose that $(Z,L)$ is either $(S^3, \text{unlink})$ or a spanning trivial relative tangle and that $\Sigma \subset Z$ is a bridge sphere or disk dividing $(Z,L)$ into tangles $(Z_1, L_1)$ and $(Z_2, L_2)$.  Let $S$ be a collection of pairwise disjoint unpunctured reducing spheres or even reducing annuli for $\Sigma$ such that no two curves of $S \cap \Sigma$ bound an unpunctured annulus in $\Sigma \setminus S$. Let $c$ be the number of closed components of $L$ and $v$ the number of arcs. Then
\[
2c + v -3 + |\boundary \Sigma| \geq |S|.
\]
\end{lemma}

\begin{proof}
Let $\Gamma$ be the dual tree to $S$ in $\Sigma$. If $\Sigma$ is a disk, let $r$ be vertex of $\Gamma$ corresponding to the component containing $\boundary \Sigma$ and consider it to be the root of $\Gamma$. If $S$ is sphere, then $\Gamma$ does not have a root. For $n = 1,2$, let $V_n$ be the number of non-root vertices of $\Gamma$ of degree $n$. Let $V_3$ be the number of non-root vertices of degree at least 3.  By the degree formula for graphs, 
\[
2|S| \geq V_1 + 2V_2 + 3V_3 + \deg(r)|\boundary \Sigma|.
\]
Since the Euler characteristic of a tree is 1, $V_1 + V_2 + V_3  = 1 + |S| - |\boundary \Sigma|$. Thus,
\[
2|S| \geq 3(1 + |S| - |\boundary \Sigma|) - V_2 - 2V_1 + \deg(r)|\boundary \Sigma|.
\]
Hence,
\[
2V_1 + V_2 + |\boundary \Sigma|(3 - \deg(r))  -3 \geq |S| 
\]

Suppose that $\Sigma' \subset \Sigma \setminus S$ corresponds to a non-root degree one vertex of $\Gamma$. Since each curve of $S \cap \Sigma$ is essential, $|L \cap \Sigma'| \geq 2$. Thus, as each component of $S$ is unpunctured, the region of $B \setminus S$ containing $\Sigma'$ contains a component of $L$. Furthermore, if it contains an arc component, then as all the annuli in $S$ are even it contains at least two arc components. Consequently, if $c_1$ (resp. $v_1$) is the number of closed components (resp. arcs) of $L$ lying in regions of $B \setminus S$ corresponding to non-root degree one vertices of $\Gamma$ then $c_1 + v_1/2 \geq V_1$. If $r$ has degree one, then the corresponding region of $B \setminus L$ contains at least one component of $L$, which could be either a closed component or an arc.

Suppose that $\Sigma' \subset \Sigma \setminus S$ corresponds to a non-root degree two vertex of $\Gamma$; let $B' \subset B \setminus S$ be the region containing $\Sigma'$.  Since no two curves of $S \cap \Sigma$ are parallel in the punctured surface $\Sigma$, $|L \cap \Sigma'| \geq 1$. Thus, $B'$ contains at least one component of $L$. Suppose that component is an arc. Let $\gamma_1$ and $\gamma_2$ be the components of $\boundary \Sigma'$ and choose the notation so that $\gamma_1$ bounds a disc in $\Sigma$ containing $\gamma_2$. Since $B'$ contains an arc, $\gamma_1$ is contained in an annulus component $A$ of $S$. Suppose, first, that $\gamma_2$ is contained in a sphere component $S' \subset S$, then $S'$ bounds a 3-ball in $B$ disjoint from all the arcs of $L$. Since $A$ is an even reducing annulus, we see that $B'$ contains at least two arcs of $L$. If $\gamma_2$ is contained in an annulus component $A'$ of $S$, then as both $A$ and $A'$ are even, we see that again $B'$ contains at least two arcs of $L$. Let $c_2$ (resp. $v_2$) be the number of closed components (resp. arcs) of $L$ lying in regions of $B \setminus S$ corresponding to non-root degree two vertices of $\Gamma$. We see that $c_2 + v_2/2 \geq V_2$. Thus, if $|\boundary \Sigma| = 0$, we have:
\[
2c + v - 3 \geq 2c_1 + v_1 + c_2 + v_2/2  - 3 \geq |S|.
\]
If $|\boundary \Sigma| = 1$ and $\deg(r) = 1$, then as mentioned there is at least one arc not contributing to $v_1 + v_2$. Thus, if $\Sigma$ is a disk, we have
\[
2c + v - 3 + 1  \geq 1 + 2c_1 + v_1 + c_2 + v_2/2  - 3 + (3 - \deg(r)) \geq |S|,
\]
as desired.
\end{proof}

  \begin{lemma}\label{lemma:commoncutdisk}
Let $(Z,L)$ be an unlink or spanning trivial relative tangle. Let $c$ be the number of closed components and $v$ the number of arc components of $L$. Suppose $\Sigma \subset Z$ is a bridge sphere or disk of bridge number $b$. If $\Sigma$ is a sphere, assume $b \geq 3$; if $\Sigma$ is a disk, assume $b \geq 2$. Let $(B,\kappa)$ and $(B', \lambda)$ be the tangles to either side of $\Sigma$. Suppose that $x \in \mc{D}_\kappa$, $y \in \mc{D}_\lambda$ are an efficient pair.  If $d(x,y) > 0$, then for any geodesic from $x$ to $y$,  there exists a common curve $s$ such that $s$ bounds cut disks in both $(B, \kappa)$ and $(B', \lambda)$.
\end{lemma}
   
\begin{proof}
Suppose $d(x,y) > 0$.  By Lemmas \ref{lemma:mindistance} and \ref{lemma:RelativeMindistance}, $d(x,y) = b - (c + v/2)$. Let $\delta = |\boundary \Sigma|$. Fix some geodesic from $x$ to $y$ and  let $\mathcal{C}$ be the collection of all common curves of the geodesic.  Since $b \geq 3-\delta$, Lemmas \ref{lemma:mindistance} and \ref{lemma:RelativeMindistance} show that $\mathcal{C}$ is nonempty.  

Since a vertex in $\mc{P}(\Sigma)$ corresponds to $2b-3 + \delta$ curves on $\Sigma$,  and adjacent vertices on the geodesic differ by exactly one curve,
\[
d(x,y)\geq 2b-3 + \delta -|\mathcal{C}|.
\]
Hence,
\[
|\mc{C}| \geq (2b - 3 + \delta) - d(x,y) = b + (c + v/2) - 3 + \delta.
\]

By Lemma \ref{consistent bounding}, each component of $\mc{C}$ either bounds a compressing disk to both sides, a spanning annulus to both sides, or bounds a cut disk to both sides. If no component of $\mc{C}$ bounds a cut disk to both sides, then by Lemma \ref{DisjointSpheres}, 
\[
2c + v - 3 + \delta \geq |\mc{C}| \geq b + c + v/2 - 3 + \delta.
\]

Now, $b \geq c + v/2$ since each closed component of $L$ intersects $\Sigma$ at least twice and each arc component at least once. Thus,
\[
2c + v - 3 + \delta \geq |\mc{C}| \geq 2c + v - 3 + \delta.
\]
Consequently,
\[
\mc{C} = 2b - 3 +\delta.
\]
This is the number of curves in $x$ (equivalently, in $y$) so $x = y$. But this implies that $d(x,y) = 0$, contradicting our hypothesis.
\end{proof}

\begin{lemma}\label{lemma:cutdisks}
Let $(Z,L)$ be an unlink or spanning relative trivial tangle. Let $c$ be the number of closed components and $v$ the number of arc components of $L$. Suppose $\Sigma \subset Z$ is a bridge sphere or disk of bridge number $b$. If $\Sigma$ is a sphere, assume $b \geq 3$; if $\Sigma$ is a disk, assume $b \geq 2$. Let $(B, \kappa)$ and $(B', \lambda)$ be the tangles to either side of $\Sigma$. Suppose that $x \in \mc{D}_\kappa$, $y \in \mc{D}_\lambda$ are an efficient pair. A curve $\gamma$ in $x$ bounds a cut disk in $(B, \kappa)$ if and only if it is also a curve in $y$ and bounds a cut disk in $(B', \lambda)$. Furthermore, such a curve is a common curve for every geodesic from $x$ to $y$. 
\end{lemma}
   
 \begin{remark}\label{distzero}
Notice that if $x$ and $y$ are an efficient defining pair with  $d(x,y) = 0$, then $x = y$ and a curve in $x = y$ bounds a cut disk to one side if and only if it bounds a cut disk to the other side by Lemma \ref{consistent bounding}. The only geodesic from $x$ to $y$ is the constant geodesic, so such curves are again common curves for any (i.e. the only) geodesic from $x$ to $y$. The difference between this situation and that when $d(x,y) > 0$ is that such a curve need not exist. If $d(x,y) > 0$, then such a curve exists by Lemma \ref{lemma:commoncutdisk}.
\end{remark}
   
\begin{proof}[Proof of Lemma \ref{lemma:cutdisks}]
Suppose, first, that $\gamma$ is a curve common to $x$ and $y$ which bounds a cut disk in $(B, \kappa)$. By Lemma \ref{consistent bounding}, it also bounds a cut disk in $(B', \lambda)$. Thus, it suffices to show that if $\gamma$ is a curve in $x$ bounding a cut disk in $(B, \kappa)$, then it also lies in $y$ and is a common curve to every geodesic from $x$ to $y$.  If $d(x,y) = 0$, then $x = y$ and the result follows trivially. Assume $d(x,y) > 0$ and that $\gamma \subset x$ is a curve bounding a cut-disk in $(B, \kappa)$. Fix some particular geodesic $\alpha$ from $x$ to $y$ in $\mc{P}(\Sigma)$.  By Lemma \ref{lemma:commoncutdisk}, there exists a common curve $s \subset \Sigma$ for the geodesic that bounds a cut-disk to both sides of $\Sigma$. If $s = \gamma$ we are done, so suppose that $s \neq \gamma$.

We induct on $b + |\boundary \Sigma|$. We consider two base cases. First, $\Sigma$ is a disk and $b=2$. Second, $\Sigma$ is a sphere and $b=3$. Suppose $\Sigma$ is a disk and  $|\Sigma \cap L| = 4$. The curve $\gamma$ must bound a disk in $\Sigma$ which contains 3 punctures. Any pants decomposition of the 4-punctured disk contains at most one such curve, so in both $x$ and $y$ that curve is $\gamma = s$, contrary to our assumption. Suppose $\Sigma$ is a sphere and $b = 3$. The curve $\gamma$ must cut $\Sigma$ into two disks each of which contains exactly three punctures. Any pants decomposition of the 6-punctured sphere contains at most one such curve. Hence, $\gamma=s$. This argument applies symmetrically to curves bounding cut disks in $(B', \lambda)$. Again, the uniqueness of $s$ in $x$ and $y$ implies the result.

Suppose $b + |\boundary \Sigma| > 2$. Surger $(Z, L)$ and $\Sigma$ along the twice-punctured sphere $S$ such that $S \cap \Sigma = s$. This results in two trivial pairs $(Z_1, L_1)$ and $(Z_2, L_2)$ containing bridge surfaces $\Sigma_1$ and $\Sigma_2$. Choose the notation so that $\boundary \Sigma \subset \boundary \Sigma_2$. Note that $Z_1 = S^3$ while $Z_2$ is either $S^3$ or $D^2 \times I$. Let $b_1$ and $b_2$ be their bridge numbers respectively and recall that $b_1 + b_2 = b + 1$. We show that we can apply our inductive hypothesis to whichever of $\Sigma_1$ or $\Sigma_2$ contains $\gamma$.

The curve $s$ bounds a disk in $\Sigma$ containing an odd number (at least 3) of punctures. If $\Sigma$ is a sphere, then it bounds two such disks. Thus, in either case, $\Sigma_1$ is a sphere with at least 4 punctures and $b_1 \geq 2$. Thus, $b_2 < b$ and $b_2 + |\boundary \Sigma_2| < b + |\boundary \Sigma|$. If $b_1 = 2$, then $\gamma$ does not lie in $\Sigma_1$, as both it and $s$ bound disks in $\Sigma$ containing an odd number (at least 3) of punctures and they are not isotopic (by assumption).

Similarly, if $\Sigma$ is a sphere, then $\Sigma_2$ is also a sphere with at least 4 punctures, $b_2 \geq 2$, and \[b_1 = b_1 + |\boundary \Sigma_1| < b + |\boundary \Sigma| = b.\] If $b_2 = 2$, then, as above, $\gamma$ does not lie in $\Sigma_2$.

If $\Sigma$ is a disk, then $\Sigma_2$ is a disk with at least two punctures and $\Sigma_1$ is a sphere. Thus, $b_2 \geq 1$ and $v_1 = 0$. Consequently, \[b_1 = b_1 + |\boundary \Sigma_1| \leq b < b +|\boundary \Sigma|.\]  If $b_2 = 1$, then $s$ bounds a once-punctured annulus in $\Sigma$ with $\boundary \Sigma$. In which case, $\gamma$ does not lie in $\Sigma_2$.  If $b_2 = 3/2$, then $s$ bounds a twice-punctured annulus in $\Sigma$ with $\boundary \Sigma$. Since both $\gamma$ and $s$ bound disks in $\Sigma$ with an odd number of punctures, again $\gamma$ does not lie in $\Sigma_2$.

We conclude that if $\gamma$ lies in $\Sigma_i$, then $b_i + |\boundary \Sigma_i| < b + |\boundary \Sigma|$, and $b_i \geq 2$ if $\Sigma_i$ is a disk and $b_i \geq 3$ if $\Sigma_i$ is a sphere.  

Let $c_1$ and $c_2$ be the number of closed components of $L_1$ and $L_2$ respectively and $v_2$ the number of arc components of $L_2$. Recall that $c_1 + (c_2 + v_2/2)= c + (v/2) + 1$. Let $d_i$ be the distances in $\mc{P}(\Sigma_i)$ between the the induced pants decompositions for $\Sigma_i$. The original geodesic path $\alpha$ from $x$ to $y$ restricts to a geodesic path $\alpha_1$ in $\mc{P}(\Sigma_1)$. Likewise, $\alpha$ restricts to a geodesic path $\alpha_2$ in $\mc{P}(\Sigma_2)$.  Since $s$ is a common curve for the geodesic from $x$ to $y$, $d = d_1 + d_2$. By Lemma \ref{lemma:mindistance} and Lemma \ref{lemma:RelativeMindistance}, $d_1\geq b_1-c_1$ and $d_2\geq b_2 - (c_2 + v_2/2)$. 

 Consequently, 
\[
b - (c + v/2)= (b_1 - c_1) + (b_2 - (c_2 + v_2/2)) \leq d_1 + d_2 = d = b - (c + v/2)
\]
Thus, $d_1= b_1-c_1$ and $d_2= b_2 - (c_2 + v_2/2)$. 

Suppose that $\gamma \subset \Sigma_i$. If $d_i = 0$, then $\gamma$ is also a common curve for $\alpha_i$ (and thus for $\alpha$), as desired. If $d_i > 0$, we apply the inductive hypothesis to conclude that $\gamma$ is again a common curve for $\alpha_i$ (and thus for $\alpha$), as desired.
\end{proof}

\section{Reducible Trisections}\label{sec:reducibletrisection}
In this section, we begin by considering the case when $\L(\mc{T}) = 0$. A similar analysis could likely be used to classify trisections with $\L \leq 2$. Afterwards, we establish a relationship between $\L$, $b(\mc{T})$, and the topology of $S$. We use the classification of bridge trisections of surfaces in $S^4$ having bridge number at most 3 by Meier and Zupan \cite[Theorem 1.8]{MZ1}. Such surfaces are unknotted. Furthermore if the bridge number is equal to 1, then the surface is an unknotted 2-sphere.  

\begin{theorem}\label{factor}
Suppose that $\mc{T}$ is a bridge trisection of a surface $S$ in $W$, with $W$ equal to $S^4$ or $B^4$, such that $\L(\mc{T}) = 0$. Then $S$ is a connected sum or a distant sum of some number of copies of unknotted spheres, unknotted projective planes, and boundary-parallel disks. Furthermore, $\mc{T}$ is the connected sum and distant sum of trisections for those surfaces.
\end{theorem}
\begin{proof}
Suppose the spine for $\mc{T}$ is  $(Z_{12},T_{12})\cup_\Sigma (Z_{23}, T_{23}) \cup_\Sigma (Z_{13}, T_{13}) \cup (V, V \cap \boundary S)$.

We induct on $b$. However, we treat the cases when $b$ is small separately since many of the results from Section 
\ref{pants} require a lower bound on $b$. 

First, suppose that $\Sigma$ is a sphere. There is a unique trisection of bridge number one, that of the unknotted sphere. If $b=2$, then  by \cite[Theorem 1.8]{MZ1} $S$ is unknotted. Moreover, for a $(b;c_1, c_2, c_3)$-bridge trisection of $S$, $-b + c_1 + c_2 + c_3 =\chi(S)$. So, $\chi(S)\geq 1$. Since a bridge number two surface can have at most two components, $S$ is either an unknotted sphere, unknotted projective plane, or the distant sum of unknotted spheres. In the latter case, $\mc{T}$ is the distant sum of bridge number one trisections.

Next, suppose $\Sigma$ is a disk. If each $T_{ij}$ has zero or one vertical arcs, then by Lemma \ref{one arc}, $\mc{T}$ is the distant sum or connected sum of a bridge trisection of the same bridge number of a surface in $S^3$ with either a bridge trisection of the empty surface in $B^3$ or a bridge trisection of bridge number 1/2 of a $\boundary$-parallel disk in $B^3$. Assume, therefore, that each $T_{ij}$ has at least two vertical arcs. In particular, $b \geq 1$.

If $b=1$, then each $T_j$ is a trivial spanning tangle. Since each $T_{ij}$ contains at least (and, therefore, exactly) two vertical arcs, every $T_j$ is a 2-strand braid and the result follows by Lemma \ref{only arcs}. If $b=3/2$, then each $T_{ij}$ is the union of 3 vertical arcs, since each $T_{ij}$ contains at least two vertical arcs. The result again follows from Lemma \ref{only arcs}.  

Henceforth, we assume that if $\Sigma$ is a sphere, then $b \geq 3$ and if $\Sigma$ is a disk, then $b \geq 2$. We also assume that the result holds for all trisections of bridge number strictly smaller than $\mc{T}$ having $\L = 0$.

Choose efficient pairs $(p^j_{ij}, p^j_{jk})$ in the disk sets for the tangles $(Z_{ij}, T_{ij})$ and $(Z_{jk}, T_{jk})$ as in the definition of $\mc{L}(\mc{T})$ so that
\[
0 = \mc{L}(\mc{T}) = d(p^1_{12}, p^2_{12}) + d(p^2_{23}, p^3_{23}) + d(p^1_{13}, p^3_{13}).
\]
We conclude that $p^1_{12} = p^2_{12}$,  $p^2_{23} = p^3_{23}$, and $p^1_{13} = p^3_{13}$. For simplicity call these points $p$, $q$, and $r$, respectively. Fix geodesics between them.

\textbf{Case 1:} $p = q = r$.

In this case, each curve $s$ of $p$ bounds a c-disk or spanning annulus in each of $(Z_{12},T_{12})$, $ (Z_{23}, T_{23})$, and $(Z_{13}, T_{13})$. By The Consistent Bounding Corollary \ref{consistent bounding cor}, $s$ satisfies one of the following:
\begin{enumerate}
\item $s$ bounds a compressing disk in each of $(Z_{12},T_{12})$, $ (Z_{23}, T_{23})$, and $(Z_{13}, T_{13})$.
\item $s$ bounds a cut disk in each of $(Z_{12},T_{12})$, $ (Z_{23}, T_{23})$, and $(Z_{13}, T_{13})$.
\item $s$ bounds a spanning annulus in each of $(Z_{12},T_{12})$, $ (Z_{23}, T_{23})$, and $(Z_{13}, T_{13})$.
\end{enumerate}

By Lemma \ref{redcriterion}, if any curve $s$ of $p$ satisfies (1) or (2), the trisection $\mc{T}$ is a distant sum or connected sum of two other trisections $\mc{T}_1$ and $\mc{T}_2$ with bridge surfaces $\Sigma_1$ and $\Sigma_2$ respectively. If $\Sigma_i$ is admissible, then the restriction $p_i$ of $p$ to the side of $s$ in $\Sigma$ corresponding to $\Sigma_i$ is nonempty. In that case, $p_i$ lies in the disk sets for all three tangles forming the spine of $\mc{T}_i$ and so $\mc{L}(\mc{T}_i) = 0$, as desired. If $\Sigma_i$ is not admissible, then $\L(\mc{T}_i) = 0$, by definition.  Recall that $b(\mc{T}_i) < b(\mc{T})$ for $i = 1,2$. Apply our inductive hypothesis to $\mc{T}_i$ to conclude that $\mc{T}$ also satisfies the conclusion of the theorem.

Suppose, therefore, that every curve $s$ of $p$ satisfies (3) and no curve of $p$ satisfies (1) or (2). This implies that each $T_{ij}$ contains only vertical arcs. By Lemma \ref{only arcs}, $S$ is the union of boundary parallel disks, each having bridge number 1/2 with respect to $\Sigma$ and the result again follows. 

\textbf{Case 2:} No two of $p$, $q$, and $r$ are equal.

Let $\alpha(pq)$, $\alpha(qr)$ and $\alpha(pr)$ be the chosen geodesics between $p$ and $q$, between $q$ and $r$, and between $p$ and $r$, respectively. Since $d(p,q) > 0$ and $d(p, r) > 0$, by Lemma \ref{lemma:cutdisks} and Remark \ref{distzero}, every curve in $p$ that bounds a cut disk in $(Z_{12}, T_{12})$ (and there is at least one such curve) is a common curve for both $\alpha(pq)$ and $\alpha(pr)$. In particular if $s$ is such a curve, then $s$ bounds a cut disk in all three of  $(Z_{12},T_{12})$, $ (Z_{23}, T_{23})$, and $(Z_{13}, T_{13})$. Observe it must also be a common curve for $\alpha(qr)$ by Lemma \ref{lemma:cutdisks}. Consequently, by Lemma \ref{redcriterion}, the result holds for $\mc{T}$, as in Case 1.

\textbf{Case 3:} Exactly two of $p$, $q$, and $r$ are equal.

Without loss of generality, assume that $p = q$ and $p \neq r$. Since $d(p,r) > 0$, by Lemma \ref{lemma:cutdisks} and Remark \ref{distzero}, there is a curve bounding a cut disk in $(Z_{12}, T_{12})$ that is a common curve for both $\alpha(pr)$ and $\alpha(qr)$ and there is at least one such curve $s$. Since $\alpha(pq)$ is the constant geodesic, $s$ is also a common curve for $\alpha(pq)$. The result now follows as in Case 2.
\end{proof}

Applying this result to a trisection witnessing $\L(S)$ and using the fact that the connected sum and distant sum of unknotted 2-spheres  and unknotted nonorientable surfaceis the distant sum of unknotted 2-spheres, we have the first theorem advertised in the Introduction.

\begin{theorem}\label{firstthm}
Suppose that $S \subset S^4$ is a smooth, closed surface with $\L(S) = 0$. Then $S$ is the distant sum of unknotted 2-spheres and unknotted nonorientable surfaces.
\end{theorem}

For the statement  and proofs of the next theorems, let $S$ be a surface in $S^4$ and consider a bridge trisection $\mc{T}$ of $S$ having bridge number $b$. Let  $\mc{S}=(Z_{12}, T_{12})\cup_{\Sigma}(Z_{23}, T_{23})\cup_{\Sigma}(Z_{31}, T_{31})$ be the spine of $\mc{T}$ and let $c_{j}$ be the number of closed components of $L_j=\alpha_{ij}\cup \alpha_{jk}$. Set $L=\mc{L}(\mc{T})$.  For each $\{i,j,k\} = \{1,2,3\}$, choose efficient pairs $(p^j_{ij}, p^j_{jk})$ for the tangles $(Z_{ij}, T_{ij})$ and $(Z_{jk}, T_{jk})$ so that
\[
L = d(p^1_{12}, p^2_{12}) + d(p^2_{23}, p^3_{23}) + d(p^1_{13}, p^3_{13}).
\]
For each $\{i,j,k\} = \{1,2,3\}$ choose a geodesic $\alpha(i)$ from $p^i_{ij}$ to $p^i_{ik}$ and a geodesic $\alpha(ij)$ from $p^i_{ij}$ to $p^j_{ij}$. Let $\mc{C}(i)$ be the set of common curves for $\alpha(i)$ and $\mc{C}(ij)$ be the set of common curves for $\alpha(ij)$.

\begin{theorem}\label{connected sum1}
Suppose that $\mc{T}$ is a bridge trisection of a smooth, closed surface $S \subset S^4$ with $b \geq 3$. If
\[
L \leq 2(c_1 + c_2 + c_3)  - 9,
\]
then $\mc{T}$ is a nontrivial connected sum or distant sum.
\end{theorem}
\begin{proof}
Since $\Sigma$ is a sphere and $b \geq 3$, the surface $\Sigma$ is admissible.

\textbf{Case 1:} There exists a curve $s$ common to two of $\mc{C}(1)$, $\mc{C}(2)$, and $\mc{C}(3)$.

Without loss of generality, suppose it to be common to $\mc{C}(1)$ and $\mc{C}(2)$. In this case, the curve $s$ is a curve of the pants decompositions $p^1_{12}$, $p^1_{13}$ and $p^2_{23}$ (as well as $p^2_{12}$). Thus, it bounds a c-disk in all three tangles $(Z_{12}, T_{12})$, $(Z_{13}, T_{13})$ and $(Z_{23}, T_{23})$. By the Consistent Bounding Corollary \ref{consistent bounding cor}, it bounds a compressing disk in all three tangles or a cut disk in all three tangles.  By Lemma \ref{redcriterion}, $\mc{T}$ is a connected sum or distant sum.

\textbf{Case 2:} There is no curve common to any two of $\mc{C}(1)$, $\mc{C}(2)$, and $\mc{C}(3)$.

As at the start of the proof of Lemma \ref{lemma:commoncutdisk}, for each $i \in \{1,2,3\}$
\[
|\mc{C}(i)| \geq b + c_i - 3.
\]
Thus, as we traverse $\alpha(ij)$, no curve of $\mc{C}(i)$ can persist to $\mc{C}(j)$. Since $\mc{C}(i)$ contains at least $b + c_i - 3$ curves and $\mc{C}(j)$ contains at least $b + c_j - 3$ curves, then the image of $\mc{C}(i)$ under $\alpha(ij)$ and $\mc{C}(j)$ must intersect in at least $c_i+c_j-3$ curves. Thus, $d(p^i_{ij}, p^j_{ij}) \geq c_i + c_j - 3.$ Consequently,
\[
L \geq 2(c_1 + c_2 + c_3)  - 9.
\]
\end{proof} 

As a corollary we have a theorem advertised in the introduction.
\begin{theorem}\label{primethm}
Suppose that $S\subset S^4$ is smooth, closed, and prime. Then
\[
\L(S) > 2b(S) + 2\chi(S) - 9.
\]
\end{theorem}
\begin{proof}
Suppose $S$ is smooth, closed, and prime. Let $\mc{T}$ be a $(b;c_1, c_2, c_3)$-bridge trisection of $S$ with $b = b(S)$ and $\L(\mc{T}) = \L(S)$. Recall that $c_1 + c_2 + c_3 =b + \chi(S)$. If $b < 3$, then by \cite{MZ2}, $S$ is unknotted, contradicting primality. Thus, $b \geq 3$. By the contrapositive of Theorem \ref{connected sum1}, we have our inequality.
\end{proof}

 \begin{theorem}\label{upperbdthm}
Suppose $b \geq 4$, and that $S$ is connected and not an unknotted 2-sphere or unknotted projective plane. If
\begin{equation}\label{upperbd}
\L(\mc{T}) \leq \frac{1}{2}(b + c_1 + c_2 + c_3) - 3,
\end{equation}
then there is a nontrivial connected sum  decomposition of $\mc{T}$ into bridge trisections $\mc{T}_1$ and $\mc{T}_2$ for surfaces $S_1$ and $S_2$ such that:
\begin{enumerate}
\item  $b(\mc{T}_1), b(\mc{T}_2) \geq 2$,
\item $\L(\mc{T}_1) + \L(\mc{T}_2) = \L(\mc{T})$, and
\item Either both $S_1$ and $S_2$ are nonorientable or $S_2$ is unknotted.
\end{enumerate}
\end{theorem}

\begin{remark}\label{MZrmk}
As remarked in \cite{MZ1}, if $\mc{T}$ is a $(b; c_1, c_2, c_3)$-trisection for a surface $S \subset S^4$, then $\chi(S) = c_1 + c_2 + c_3 - b$.  Thus,
\[
b + c_1 + c_2 + c_3 = \chi(S) + 2b.
\]
In particular, if $S$ is orientable, then the right hand side of Inequality \eqref{upperbd} is an even integer.
\end{remark}

\begin{proof}
Meier and Zupan showed that every bridge trisection of bridge number at most 3 is a bridge trisection of an unknotted surface \cite[Theorem 1.8]{MZ1}. We prove our result by induction on $b$, handling the base case and inductive case simultaneously.  Assume that $b \geq 4$ and that
\[
L \leq \frac{1}{2}(b + c_1 + c_2 + c_3) - 3.
\]
Note that $\Sigma$ is admissible and that $\mc{T}$ cannot be a distant sum as $S$ is connected. Assume also that the result holds for any trisection with bridge number at least 4, strictly less than $b$, and satisfying the corresponding Inequality \eqref{upperbd}.

\textbf{Case 1:} There is a curve $s$ common to all six of $\mc{C}(1)$, $\mc{C}(2)$, $\mc{C}(3)$, $\mc{C}(12)$, $\mc{C}(23)$, and $\mc{C}(13)$. 

By the Consistent Bounding Corollary \ref{consistent bounding cor}, $s$ bounds a cut disk in all three tangles $(Z_{ij}, T_{ij})$ for distinct $i,j \in \{1,2,3\}$ or that curve bounds a compressing disk in all three tangles. By Lemma \ref{redcriterion}, $\mc{T}$  is a nontrivial connected sum with factors $\mc{T}_1$ and $\mc{T}_2$. For $i = 1,2$, the bridge number $b(\mc{T}_i)$ satisfies $2 \leq b(\mc{T}_i) < b$ since the connected sum is nontrivial and by the properties of connected sum. If $b(\mc{T}_i) \leq 3$, for $i = 1$ or $i = 2$, then by Meier and Zupan's result, we are done. So also suppose that $b(\mc{T}_i) \geq 4$, for $i = 1,2$. This implies that the trisection surface $\Sigma_i$ for $\mc{T}_i$ is admissible. Consequently, each of the six geodesics $\alpha(i)$ and $\alpha(ij)$ for distinct $i,j \in \{1,2,3\}$, restrict to geodesics in $\mc{P}(\Sigma_i)$. Hence, $\L(\mc{T}_1) + \L(\mc{T}_2) = \L(\mc{T})$. 

Suppose that $\mc{T}_i$ is a $(b_i; x_i, y_i, z_i)$-trisection.  Recall that 
\[\begin{array}{rcl}
b &=& b_1 + b_2 - 1\\
c_1 &=& x_1 + x_2 - 1 \\
c_2 &=& y_1 + y_2 - 1 \\
c_3 &=& z_1 + z_2 - 1.
\end{array}\]
For $i = 1,2$, let $B_i = \frac{1}{2}(b_i + x_i + y_i + z_i) - 3$. By Remark \ref{MZrmk}, for orientable surfaces it is a positive integer. We have
\[
\L(\mc{T}_1) + \L(\mc{T}_2) \leq B_1 + B_2 + 1.
\]
Since $\L$ is an integer, either both $S_1$ and $S_2$ are non-orientable (with odd Euler characteristic) or  $\L(\mc{T}_1) \leq B_1$ or $\L(\mc{T}_2) \leq B_2$. If both $S_1$ and $S_2$ are non-orientable, we have our theorem. If not, the result follows by applying the inductive hypothesis to whichever $\mc{T}_i$ satisfies the inequality.

\textbf{Case 2:} There is no curve common to all six of $\mc{C}(1)$, $\mc{C}(2)$, $\mc{C}(3)$, $\mc{C}(12)$, $\mc{C}(23)$, and $\mc{C}(13)$. 

Let $\ell(i) = b - c_i$ be the length of $\alpha(i)$ and $\ell(ij) = d(p^i_{ij}, p^j_{ij})$ be the length of $\alpha(ij)$. Recall that $L = \ell(12) + \ell(23) + \ell(13)$. 

Consider a curve $s$ of $p^1_{12}$ as we traverse the loop $\Lambda$ defined by the geodesics $\alpha(1)$, $\alpha(13)$, $\alpha(3)$, $\alpha(23)$, and $\alpha(2)$. By the hypothesis of the case, the curve $s$ must be moved as we traverse one of the edges of the loop. When it does so for the first time, it becomes a new loop $s' \subset \Sigma$ intersecting $s$ exactly twice. At the conclusion of our traversal of the loop, we arrive back at $p^1_{12}$, a collection of pairwise disjoint curves. Thus, $s'$ must be moved again as we traverse some subsequent edge of the loop. As $p^1_{12}$ has $2b-3$ curves,
\[
\ell(1) + \ell(2) + \ell(3) + L \geq 2(2b-3).
\]
Consequently,
\[
3b - (c_1 + c_2 + c_3) + L\geq 4b - 6.
\]
Thus,
\[
L \geq b + (c_1 + c_2 + c_3) - 6.
\]
By Inequality \eqref{upperbd},
\[
\frac{1}{2}(b + c_1 + c_2 + c_3) - 3 \geq L \geq b + (c_1 + c_2 + c_3) - 6.
\]
Thus, $L = 0$ and so by Theorem \ref{factor}, $S$ is an unknotted 2-sphere or nonorientable surface and $\mc{T}$ is the connected and distant sum of trisections of unknotted 2-spheres and projective planes. In the latter case, if there is more than one projective plane in the sum, we see that the result still holds.
\end{proof}

\begin{theorem}\label{reduciblethm}
If $S \subset S^4$ is a smooth, closed, connected, orientable, irreducible surface then
\[
\L(S) > b(S) - g(S) - 2,
\]
where $g(S)$ is the genus of $S$.
\end{theorem}

\begin{proof}
Suppose that $S$ is smooth, closed, connected, orientable, and irreducible. By \cite[Theorem 1.8]{MZ1}, $b= b(S) \geq 4$. Let $\mc{T}$ be a $(b;c_1,c_2, c_3)$-trisection of $S$ such that $\L(\mc{T}) = \mc{L}(S)$. Recall that $2-2g(S) = \chi(S) = c_1 + c_2 + c_3 - b$. If $\mc{T}$ is a nontrivial connected sum with factors $\mc{T}_1$ and $\mc{T}_2$ that are bridge trisections of surfaces $S_1$ and $S_2$ with $S_2$ being trivial, then $S_2$ must be an unknotted sphere, since $S$ is irreducible and orientable. But the fact that the connected sum is nontrivial contradicts the choice of $\mc{T}$ to satisfy $b(\mc{T}) = b(S)$. Thus, by Theorem \ref{upperbdthm}, 
\[\begin{array}{rcl}
\L(S) &=& \L(\mc{T}) \\
&>& \frac{1}{2}(b + c_1 + c_2 + c_3) - 3 \\
&=& \frac{1}{2}(2b + 2 - 2g(S)) - 3 \\
& =& b - g(S) - 2.
\end{array} 
\]
\end{proof}

   \bibliographystyle{plain}
\bibliography{Linvariant}

\end{document}